\pdfoutput=1
\documentclass[11pt,a4paper]{article}

\usepackage[T1]{fontenc}
\usepackage[utf8]{inputenc}
\usepackage{lmodern}
\usepackage[a4paper,margin=32mm]{geometry}
\usepackage[protrusion=true,expansion=false]{microtype}
\usepackage{mathtools}
\usepackage{amssymb}
\usepackage{amsthm}
\usepackage[nointegrals]{wasysym}
\usepackage{tikz-cd}
\usetikzlibrary{arrows,positioning}
\usepackage{enumitem}
\usepackage{url}
\usepackage[numbers,sort&compress]{natbib}
\usepackage[hidelinks]{hyperref}

\allowdisplaybreaks
\providecommand{\defeq}{\coloneqq}
\providecommand{\Cat}{\mathbf{Cat}}
\providecommand{\Set}{\mathbf{Set}}
\providecommand{\Pos}{\mathbf{Pos}}
\providecommand{\Top}{\mathbf{Top}}
\providecommand{\ConTop}{\mathbf{ConTop}}
\providecommand{\LocConTop}{\mathbf{LocConTop}}
\providecommand{\ConCat}{\mathbf{ConCat}}
\providecommand{\Qbs}{\mathbf{Qbs}}
\providecommand{\PsTop}{\mathbf{PsTop}}
\providecommand{\wCpo}{\mathbf{wCpo}}
\providecommand{\FinSet}{\mathbf{FinSet}}
\providecommand{\FinFam}{\mathbf{FinFam}}
\providecommand{\FinDist}[1]{\mathbf{FinDist}(#1)}
\providecommand{\CoProdCat}{\mathbf{CoProdCat}}
\providecommand{\ProdCat}{\mathbf{ProdCat}}
\providecommand{\DistCat}{\mathbf{DistCat}}
\providecommand{\NN}{\mathbb{N}}
\providecommand{\RR}{\mathbb{R}}
\providecommand{\catA}{\mathcal{A}}
\providecommand{\catB}{\mathcal{B}}
\providecommand{\catC}{\mathcal{C}}
\providecommand{\catD}{\mathcal{D}}
\providecommand{\catE}{\mathcal{E}}
\providecommand{\catV}{\mathcal{V}}
\providecommand{\jd}{d}
\providecommand{\Fam}{\mathsf{Fam}}
\providecommand{\IFam}{\mathfrak{Fam}}
\providecommand{\Dist}{\mathsf{Dist}}
\providecommand{\Par}{\mathsf{Par}}
\providecommand{\terminal}{1}
\providecommand{\initial}{0}
\providecommand{\id}{\mathrm{id}}
\providecommand{\sumfam}[2]{[#1\mid #2]}
\providecommand{\prodfam}[2]{\langle #1\mid #2\rangle}
\DeclareMathOperator{\Ob}{Ob}
\newcommand{\Loc}{\mathbf{Loc}}
\newcommand{\Frm}{\mathbf{Frm}}

\theoremstyle{plain}
\newtheorem{theorem}{Theorem}[section]
\newtheorem{proposition}[theorem]{Proposition}
\newtheorem{lemmma}[theorem]{Lemma}
\newtheorem{corollary}[theorem]{Corollary}
\newtheorem{openquestion}[theorem]{Open Question}

\theoremstyle{definition}
\newtheorem{definition}[theorem]{Definition}

\theoremstyle{remark}
\newtheorem{example}[theorem]{Example}
\newtheorem{remark}[theorem]{Remark}
\newtheorem{counterexample}[theorem]{Counterexample}
\numberwithin{equation}{section}

\begin{document}

\title{Free Doubly-Infinitary Distributive Categories are Cartesian Closed}
\author{Fernando Lucatelli Nunes\textsuperscript{1}\thanks{Corresponding author: \texttt{fln@uc.pt}.}
  \and Matthijs V\'ak\'ar\textsuperscript{2}\thanks{\texttt{m.i.l.vakar@uu.nl}.}}
\date{}

\maketitle

\begin{center}
\small
\textsuperscript{1}CMUC, Department of Mathematics, University of Coimbra,
3000-143 Coimbra, Portugal\\
\textsuperscript{2}Department of Information and Computing Sciences,
Utrecht University, Netherlands
\end{center}

\begin{abstract}
	We study the composite free completion
	\[
	\Dist(\catC)\defeq \Fam\!\bigl(\Fam(\catC^{op})^{op}\bigr),
	\]
	obtained by first freely adjoining small products and then freely adjoining small coproducts. A natural pseudodistributive law equips this endo-pseudofunctor with a composite pseudomonad structure. Its pseudoalgebras are precisely the categories with small products and small coproducts in which small products distribute over small coproducts. We call such categories \emph{doubly-infinitary distributive}. This condition is natural, but does not seem to have been systematically isolated in the literature. Thus $\Dist(\catC)$ is the free doubly-infinitary distributive category on $\catC$. Our main result is that $\Dist(\catC)$ is \emph{cartesian closed}. Finally, we compare doubly-infinitary distributivity with extensivity, ordinary infinitary distributivity, and cartesian closedness by means of separating examples.
\end{abstract}


\medskip


\medskip

\section{Introduction}
Many categorical structures are governed not only by the existence of limits and colimits, but by the way these constructions interact. Distributivity phenomena provide a basic family of such interactions, ranging from compatibility between finite products and finite coproducts to stronger forms involving more general limits and colimits. Classical accounts include finitary and infinitary distributive categories~\cite{carboni1993introduction}, as well as stronger variants such as completely distributive categories~\cite{marmolejo2012completely}. The present paper isolates and studies a natural intermediate condition in this landscape: categories with small products and small coproducts in which small products distribute over small coproducts.

We begin by recalling the familiar distributivity conditions in a form that makes the comparison with our setting immediate.

\subsection*{Distributivity conditions}

Let $\catC$ be a category with finite products and finite coproducts. It is \emph{(finitary) distributive} if, for all objects $A,B,C\in\catC$, the canonical comparison map
\begin{equation}\label{eq:canonical-comparison-finitary}
	(A\times B)\sqcup (A\times C)\ \xrightarrow{\ \cong\ }\ A\times (B\sqcup C)
\end{equation}
is invertible. Equivalently, for every object $A$, the product functor
\begin{equation}\label{eq:product-for-A}
	A\times -:\catC\to\catC
\end{equation}
preserves finite coproducts.

If $\catC$ has all small coproducts, one may strengthen this to \emph{infinitary} distributivity by requiring that, for every set $I$ and every family $(B_i)_{i\in I}$, the canonical comparison
\begin{equation}\label{eq:canonical-comparison-infinitary}
	\coprod_{i\in I}(A\times B_i)\ \xrightarrow{\ \cong\ }\ A\times\left(\coprod_{i\in I}B_i\right)
\end{equation}
is invertible; equivalently,~\eqref{eq:product-for-A} preserves all small coproducts.

The present paper studies a different strengthening, in which \emph{small products distribute over small coproducts}. The basic comparison has the form
\begin{equation}\label{eq:canonical-morphism-that-defines-intro}
	\coprod_{r\in\prod_{j\in J}I_j}\ \prod_{j\in J}C_{j,r_j}
	\ \longrightarrow\
	\prod_{j\in J}\ \coprod_{i\in I_j}C_{j i},
\end{equation}
for a family of families $(C_{j i})_{j\in J,\ i\in I_j}$. The formal definition is given later in Definition~\ref{def:doubly-infinitary-distributivity}, together with its pseudomonadic interpretation. We call the categories satisfying this condition \emph{doubly-infinitary distributive categories}, and show that they are precisely the pseudoalgebras for the composite pseudomonad $\Dist$ studied in this paper. In the presence of the relevant completeness and cocompleteness, this condition lies strictly between ordinary infinitary distributivity and complete distributivity, as the examples below show. Although natural, it does not seem to have been systematically isolated in the literature; one contribution of this paper is to make this condition explicit.

The rest of the paper studies the free categorical construction naturally associated with this comparison. We first treat it concretely, as an explicit composite of free product and coproduct completions. Only later do we return to the pseudomonadic statement identifying its pseudoalgebras with the categories satisfying the above distributivity condition.

\subsection*{The free doubly-infinitary distributive category}

We denote by $\Fam(\catC)$ the free completion of $\catC$ under small coproducts. Dually, $\Fam(\catC^{op})^{op}$ is the free completion of $\catC$ under small products. The central concrete construction of the paper is the composite
\[
\Dist(\catC)\defeq \Fam\!\bigl(\Fam(\catC^{op})^{op}\bigr),
\]
whose objects may be viewed as formal coproducts of formal products of objects of $\catC$.

For the moment, $\Dist(\catC)$ should be understood as this explicit composite free completion. The terminology is justified later by the pseudomonadic analysis: the free product-completion and free coproduct-completion pseudomonads are related by a canonical pseudodistributive law, and the resulting composite pseudomonad has as its pseudoalgebras precisely the \textit{doubly-infinitary distributive categories} mentioned above. Thus $\Dist(\catC)$ is, in this sense, the free doubly-infinitary distributive category generated by $\catC$.

The main theorem is proved before the pseudomonadic verification, since it only uses the explicit description of $\Dist(\catC)$ as a category of formal coproducts of formal products. By construction, $\Dist(\catC)$ has small coproducts. By a well-known result on products in free coproduct completions, since the inner completion $\Fam(\catC^{op})^{op}$ has small products, $\Dist(\catC)$ also has small products. The additional structure proved in this paper is the following.

\medskip
\noindent\textbf{Main result.}
\emph{For every category $\catC$, the category $\Dist(\catC)$ is cartesian closed.}  Moreover, exponentials admit an explicit description (Theorem~\ref{thm:ccc}).

\medskip
The proof rests on a structural feature of the free product completion $\Fam(\catC^{op})^{op}$. The objects coming from $\catC$ satisfy a strong atomicity property: they are coconnected, or coindecomposable, in a suitable sense. This implies a form of \emph{multiexponentiation}: by denoting $\catD\defeq\Fam(\catC^{op})^{op}$, for fixed $Y$ and $Z$, the presheaf $\catD(-\times Y,Z)$ need not be representable, but it is a small coproduct of representables. Passing from $\catD$ to $\Fam(\catD)$ turns these familial representations into genuine exponentials. This proves cartesian closedness of $\Dist(\catC)$ and yields the explicit ``partial map'' formula for exponentials.

\subsection*{Outline and background}

The paper is organised as follows. Section~\ref{sec:FAM} recalls the free completion under small coproducts and its dual free completion under small products, and introduces the composite completion $\Dist(\catC)$. Section~\ref{sec:main-result} proves directly that $\Dist(\catC)$ is cartesian closed and gives an explicit formula for its exponentials.

Section~\ref{sec:pseudomonad-Dist} gives the structural explanation for the notation. It defines doubly-infinitary distributive categories, describes the relevant pseudodistributive law, and proves that the pseudoalgebras for $\Dist$ are precisely the doubly-infinitary distributive categories. Consequently, the category $\Dist(\catC)$ studied in Section~\ref{sec:main-result} is the free doubly-infinitary distributive category generated by $\catC$.

Section~\ref{sec:discrete-case} computes the free doubly-infinitary distributive completion on coproducts of categories, with special attention to the case of discrete categories. It also records a complementary bicategorical viewpoint, namely a biproduct phenomenon in the 2-category of categories with products. Section~\ref{sec:examples} then gives examples and counterexamples separating doubly-infinitary distributivity from extensivity, ordinary infinitary distributivity, and cartesian closedness.

Finally, Section~\ref{sec:final-remarks} discusses further properties of $\Dist(\catC)$, the role of non-canonical isomorphisms, open questions about generalised categorical structures, and the relation with G\"odel's Dialectica interpretation.

For background on the $\Fam$-construction and free product and coproduct completions, see~\cite{zbMATH05306928, zbMATH07186728, nunes2023chad, zbMATH07799814}. For pseudomonads, pseudoalgebras, and pseudodistributive laws, see~\cite{zbMATH04105188, MR3491845, zbMATH06970806, zbMATH02116176, zbMATH05256222, marmolejo1999distributive}.

\subsection*{Related work}
A closely related distributivity condition appears in the study of $\Pi$- and $\Sigma$-types in fibrations over a locally cartesian closed base $\mathbb{B}$; see, for example, von Glehn~\cite{von2018polynomials} and Glehn--Moss~\cite{glehnmoss2018}. Our focus is the special case of fibrations $\Fam(\catC)\to\Set$, where $\Pi$- and $\Sigma$-types reduce to products and coproducts in $\catC$; see~\cite[Theorem~3.5.2]{vakar2017search} or~\cite[Theorem~12]{vakar2015categorical}.

\section{Free product and coproduct completions}\label{sec:FAM}

We recall the free completion under small coproducts and its dual free completion under small products. Combining these constructions gives the composite completion $\Dist(\catC)$ whose cartesian closedness is proved in Section~\ref{sec:main-result}. Its pseudomonadic interpretation, including the formal definition of \textit{doubly-infinitary distributive categories}, is postponed to Section~\ref{sec:pseudomonad-Dist}.

The $\Fam$-construction is classical; see~\cite{zbMATH07186728, zbMATH05306928, carboni1993introduction, nunes2023chad, zbMATH07799814}.

\subsection{The category \texorpdfstring{$\Fam(\catC)$}{Fam(C)}}

Let $\jd:\Set\to\Cat$ denote the inclusion sending a set to the corresponding discrete category. For a category $\catC$, define the indexed category
\[
\IFam(\catC):\Set^{op}\to\Cat,
\qquad
\IFam(\catC)(I)\defeq \Cat(\jd(I),\catC)\cong \catC^{I},
\]
and, for a function $u:J\to I$, set $\IFam(\catC)(u)\defeq u^{*}$, reindexing along $\jd(u)$.

The Grothendieck construction $\int \IFam(\catC)$ yields a split fibration over $\Set$ whose total category is the \emph{family category} $\Fam(\catC)$. Concretely:

\begin{itemize}
	\item An object is a pair $(I,C)$ where $I$ is a set and $C:\jd(I)\to\catC$ is a functor, equivalently an $I$-indexed family $(C_i)_{i\in I}$ of objects of $\catC$. We often denote this object by the formal coproduct
	\[
	\sumfam{C_i}{i\in I}.
	\]
	
	\item A morphism $(I,C)\to (J,C')$ consists of a function $f:I\to J$ together with a family of arrows $(\varphi_i:C_i\to C'_{f(i)})_{i\in I}$ in $\catC$. Equivalently, it is a natural transformation $\varphi:C\Rightarrow C'\circ \jd(f)$.
\end{itemize}

Thus, for family objects $\sumfam{C_i}{i\in I}$ and $\sumfam{C'_j}{j\in J}$, one has
\[
\Fam(\catC)\Bigl(\sumfam{C_i}{i\in I},\,\sumfam{C'_j}{j\in J}\Bigr)
\ \cong\
\coprod_{f\in\Set(I,J)}\ \prod_{i\in I}\catC\bigl(C_i,\,C'_{f(i)}\bigr).
\]
Identities and composition are defined pointwise: the identity on $\sumfam{C_i}{i\in I}$ is
\[
(\id_I,(\id_{C_i})_{i\in I}),
\]
and if
\[
(f,(\varphi_i)_{i\in I}):\sumfam{C_i}{i\in I}\to\sumfam{C'_j}{j\in J}
\]
and
\[
(g,(\psi_j)_{j\in J}):\sumfam{C'_j}{j\in J}\to\sumfam{C''_k}{k\in K},
\]
then their composite is
\[
(g\circ f,(\psi_{f(i)}\circ \varphi_i)_{i\in I}).
\]

\subsection{Universal property and the dual completion}

The assignment $\catC\mapsto \Fam(\catC)$ extends to a lax-idempotent, or Kock--Z\"oberlein, pseudomonad on $\Cat$; see~\cite{kock1995monads, zbMATH01024330, zbMATH03680046, clementino2023lax}. Its unit
\[
\eta^\Sigma_\catC:\catC\to\Fam(\catC)
\]
sends an object $C\in\catC$ to the singleton family $\sumfam{C}{\{\ast\}}$, and sends a morphism $f:C\to D$ in $\catC$ to the morphism
\[
(\id_{\{\ast\}},(f)):\sumfam{C}{\{\ast\}}\to \sumfam{D}{\{\ast\}}.
\]

The multiplication
\[
\mu^\Sigma_\catC:\Fam(\Fam(\catC))\to\Fam(\catC)
\]
flattens a family of families:
\[
\mu^\Sigma_\catC\!\left(\sumfam{\,\sumfam{C_{j,i}}{i\in I_j}\,}{j\in J}\right)
\ =
\sumfam{C_{j,i}}{(j,i)\in \coprod_{j\in J} I_j}.
\]
On morphisms, $\mu^\Sigma_\catC$ acts by flattening the index functions. Concretely, a morphism in $\Fam(\Fam(\catC))$
\[
\Bigl(f,\ (h_j,(\varphi_{j,i})_{i\in I_j})_{j\in J}\Bigr):
\sumfam{\,\sumfam{C_{j,i}}{i\in I_j}\,}{j\in J}
\longrightarrow
\sumfam{\,\sumfam{C'_{j',i'}}{i'\in I'_{j'}}\,}{j'\in J'}
\]
consists of a function $f:J\to J'$ together with, for each $j\in J$, a morphism
\[
(h_j,(\varphi_{j,i})_{i\in I_j}):
\sumfam{C_{j,i}}{i\in I_j}\longrightarrow \sumfam{C'_{f(j),i'}}{i'\in I'_{f(j)}}
\]
in $\Fam(\catC)$. Its image under $\mu^\Sigma_\catC$ has index function
\[
\bar f:\coprod_{j\in J} I_j\longrightarrow \coprod_{j'\in J'} I'_{j'},
\qquad
\bar f(j,i)\defeq (f(j),h_j(i)),
\]
and component maps $\varphi_{j,i}:C_{j,i}\to C'_{f(j),h_j(i)}$.

If $\catC$ has small coproducts, there is a coproduct functor
\[
\coprod:\Fam(\catC)\to\catC,
\qquad
\sumfam{C_i}{i\in I}\longmapsto \coprod_{i\in I}C_i,
\]
and this functor is left adjoint to the unit $\eta^\Sigma_\catC$, by the defining hom-formula for $\Fam(\catC)$. The corresponding $2$-category of pseudoalgebras and pseudomorphisms is equivalent to the $2$-category $\CoProdCat$ of categories with small coproducts, coproduct-preserving functors, and natural transformations.

\begin{remark}[A lali witness for the Kock--Z\"oberlein property]\label{rem:Fam-KZ}
	Observe that the counit
	\[
	\coprod_\catC\eta^\Sigma_\catC\Rightarrow \id_\catC
	\]
	is the identity, since the coproduct of a singleton family is chosen to be the object itself with coprojection $\id$. Hence $\coprod_\catC\dashv \eta^\Sigma_\catC$ is a \emph{lali} adjunction; see, for instance,~\cite[0.3.B]{MR0213413} and~\cite[1.1]{clementino2023lax}. This is a characterisation of the Kock--Z\"oberlein, or lax-idempotent, condition for $\Fam$; see, for example,~\cite[Theorem~3.15]{clementino2023lax}.
\end{remark}

The composite
\[
\Cat \xrightarrow{op} \Cat^{co} \xrightarrow{\Fam(-)} \Cat^{co} \xrightarrow{op} \Cat
\]
inherits a colax-idempotent pseudomonad structure whose pseudoalgebras are categories with small products. We denote
\[
\Fam_{\Sigma}(\catC)\defeq \Fam(\catC),
\qquad
\Fam_{\Pi}(\catC)\defeq \Fam(\catC^{op})^{op}.
\]
Objects of $\Fam_{\Pi}(\catC)$ are denoted by formal products
\[
\prodfam{C_i}{i\in I},
\]
and a morphism
\[
\prodfam{C_i}{i\in I}\to \prodfam{D_{i'}}{i'\in I'}
\]
consists of a function $f:I'\to I$ and a family $(\alpha_{i'}:C_{f(i')}\to D_{i'})_{i'\in I'}$ in $\catC$.

\subsection{The composite completion and inclusions}

Define the composite completion
\begin{equation}\label{EQ:DIST}
	\Dist(\catC)\defeq \Fam_{\Sigma}\!\bigl(\Fam_{\Pi}(\catC)\bigr)
	=\Fam\!\bigl(\Fam(\catC^{op})^{op}\bigr).
\end{equation}
Thus an object of $\Dist(\catC)$ is a formal coproduct of formal products of objects of $\catC$.

The notation $\Dist$ is justified by the pseudomonadic result proved in Section~\ref{sec:pseudomonad-Dist}: the free product-completion and free coproduct-completion pseudomonads are related by a canonical pseudodistributive law, and the resulting composite pseudomonad has as its pseudoalgebras precisely the \textit{doubly-infinitary distributive categories} as introduced therein; see Definition~\ref{def:doubly-infinitary-distributivity}. Consequently, by the structural analysis of Section~\ref{sec:pseudomonad-Dist}, $\Dist(\catC)$ is the \emph{free doubly-infinitary distributive category} generated by $\catC$.

We use the following standard inclusions, induced by pseudomonad units:
\[
\widehat{(-)}:\catC\to\Fam_{\Pi}(\catC),
\qquad
\underline{(-)}:\Fam_{\Pi}(\catC)\to\Dist(\catC),
\qquad
\overline{(-)}\defeq \underline{(\widehat{(-)})}:\catC\to\Dist(\catC).
\]
Concretely, $\widehat{C}$ is the singleton product $\prodfam{C}{\{\ast\}}$ and $\overline{C}$ is the singleton coproduct of that singleton product.

When convenient, we suppress the $\sumfam{-}{-}$ and $\prodfam{-}{-}$ delimiters and denote objects of $\Dist(\catC)$ by
\[
\sum_{j\in J}\prod_{i\in I_j} C_{j,i},
\]
always understood as a formal coproduct, indexed by $J$, of formal products, indexed by the sets $I_j$.

\section{\texorpdfstring{$\Dist(\catC)$}{Dist(C)} is cartesian closed}\label{sec:main-result}

We now prove the main theorem: for every category $\catC$, the concrete composite completion
\[
\Dist(\catC)=\Fam\bigl(\Fam(\catC^{op})^{op}\bigr)
\]
is cartesian closed. The proof is a direct analysis of the category of formal coproducts of formal products. After Section~\ref{sec:pseudomonad-Dist}, the same result may be read as saying that \textit{the free doubly-infinitary distributive category on $\catC$ is cartesian closed.}

The result follows from a multiexponentiation phenomenon in the free product completion $\Fam(\catC^{op})^{op}$, together with the fact that the outer $\Fam$-construction turns coproducts of representables into representable functors.

We begin by recalling the explicit objects, morphisms, products, and coproducts of $\Dist(\catC)$.

\subsection{Objects and morphisms}

An object of $\Dist(\catC)=\Fam(\Fam(\catC^{op})^{op})$ is a family in $\Fam_{\Pi}(\catC)$, hence may be denoted by
\[
\sumfam{\prodfam{C_{j,i}}{i\in I_j}}{j\in J}.
\]
Equivalently, it is a set $J$ together with, for each $j\in J$, a set $I_j$ and an $I_j$-indexed family $(C_{j,i})_{i\in I_j}$ of objects of $\catC$.

A morphism
\[
\sumfam{\prodfam{C_{j,i}}{i\in I_j}}{j\in J}
\longrightarrow
\sumfam{\prodfam{C'_{j',i'}}{i'\in I'_{j'}}}{j'\in J'}
\]
is, by unwinding the two $\Fam$-constructions, the following data:
\begin{enumerate}
\item a function $g:J\to J'$;
\item for each $j\in J$, a function $f_j:I'_{g(j)}\to I_j$;
\item for each $j\in J$ and each $i'\in I'_{g(j)}$, a morphism in $\catC$
\[
\alpha^{j}_{i'}:\ C_{j,f_j(i')}\longrightarrow C'_{g(j),i'}.
\]
\end{enumerate}
We may therefore denote such a morphism by a triple
\begin{equation}\label{eq:morphism-FamFamop}
(g,(f_j)_{j\in J},(\alpha^{j}_{i'})_{j,i'}).
\end{equation}
The identity on $\sumfam{\prodfam{C_{j,i}}{i\in I_j}}{j\in J}$ is
\[
\bigl(\id_J,(\id_{I_j})_{j\in J},(\id_{C_{j,i}})_{j\in J,\ i\in I_j}\bigr).
\]
Given also a morphism
\[
(h,(k_{j'})_{j'\in J'},(\beta^{j'}_{i''})_{j',i''}):
\sumfam{\prodfam{C'_{j',i'}}{i'\in I'_{j'}}}{j'\in J'}
\to
\sumfam{\prodfam{C''_{j'',i''}}{i''\in I''_{j''}}}{j''\in J''},
\]
their composite is
\[
\bigl(h\circ g,\ (f_j\circ k_{g(j)})_{j\in J},\ (\beta^{g(j)}_{i''}\circ \alpha^{j}_{k_{g(j)}(i'')})_{j\in J,\ i''\in I''_{h(g(j))}}\bigr),
\]
with composition in $\catC$ taken componentwise.

In particular, the hom-set admits the explicit form
\begin{align*}
&\Dist(\catC)\Bigl(
\sumfam{\prodfam{C_{j,i}}{i\in I_j}}{j\in J},\
\sumfam{\prodfam{C'_{j',i'}}{i'\in I'_{j'}}}{j'\in J'}
\Bigr)
\\
&\qquad\cong\
\coprod_{g\in \Set(J,J')}
\ \prod_{j\in J}
\ \coprod_{f_j\in \Set(I'_{g(j)}, I_j)}
\ \prod_{i'\in I'_{g(j)}}
\ \catC\bigl(C_{j,f_j(i')},\,C'_{g(j),i'}\bigr).
\end{align*}

\subsection{Products and coproducts}

Since $\Dist(\catC)=\Fam(\catD)$ for $\catD=\Fam(\catC^{op})^{op}$, coproducts are adjoined freely and are computed by disjoint union of the outer index sets.

\begin{lemmma}[Coproducts in $\Dist(\catC)$]\label{lem:coproducts}
Let $(A_k)_{k\in K}$ be a family of objects of $\Dist(\catC)$, with
\[
A_k=\sumfam{\prodfam{\,{}^{k}C_{j,i}\,}{i\in {}^{k}I_j}}{j\in {}^{k}J}.
\]
Then the coproduct exists and is given by
\[
\coprod_{k\in K} A_k
\ \cong\
\sumfam{\prodfam{\,{}^{k}C_{j,i}\,}{i\in {}^{k}I_j}}{(k,j)\in \coprod_{k\in K} {}^{k}J}.
\]
\end{lemmma}

\begin{proof}
This is the standard coproduct formula in $\Fam(\catD)$, applied with $\catD=\Fam(\catC^{op})^{op}$.
\end{proof}

Products in $\Fam(\catD)$ exist whenever $\catD$ has products (see, e.g.,~\cite{zbMATH05306928, zbMATH07186728, nunes2023chad}). Since $\catD=\Fam(\catC^{op})^{op}$ is the free completion under small products, it has small products, and hence so does $\Dist(\catC)$.

\begin{lemmma}[Products in $\Dist(\catC)$]\label{lem:products}
Let $(A_k)_{k\in K}$ be a family of objects of $\Dist(\catC)$, with
\[
A_k=\sumfam{\prodfam{\,{}^{k}C_{j,i}\,}{i\in {}^{k}I_j}}{j\in {}^{k}J}.
\]
Then the product exists and is given by
\[
\prod_{k \in K} A_k
\ \cong\
\sumfam{
\prodfam{\,{}^{k}C_{j_k,i}\,}{(k,i)\in \coprod_{k\in K} {}^{k}I_{j_k}}
}{
(j_k)_{k\in K}\in \prod_{k\in K} {}^{k}J
}.
\]
\end{lemmma}

\begin{proof}
In $\Fam(\catD)$ one has
\[
\prod_{k\in K}\sumfam{D^{k}_{j}}{j\in {}^{k}J}
\ \cong\
\sumfam{\prod_{k\in K} D^{k}_{j_k}}{(j_k)_{k\in K}\in \prod_{k\in K}{}^{k}J},
\]
with products computed in $\catD$.
Here $\catD=\Fam(\catC^{op})^{op}$, and products in $\catD$ are concatenations of index sets:
\[
\prod_{k\in K}\prodfam{\,{}^{k}C_{j_k,i}\,}{i\in {}^{k}I_{j_k}}
\ \cong\
\prodfam{\,{}^{k}C_{j_k,i}\,}{(k,i)\in\coprod_{k\in K}{}^{k}I_{j_k}}.
\]
Substituting this into the product formula in $\Fam(\catD)$ yields the stated expression.
\end{proof}

\subsection{Multiexponentiation and exponentials}

Although $\Dist(\catC)$ is constructed by freely adjoining products and coproducts, it also turns out to be cartesian closed. The proof is cleanest when expressed via a notion of ``familial representability''.

\subsubsection*{Exponentiating, multiexponentiating and coconnected objects}

Let $\catD$ be a category with small products. For $Y,Z\in\catD$, consider the presheaf
\[
\catD(-\times Y,Z):\catD^{op}\to\Set.
\]
We say that $\catD(-\times Y,Z)$ is familially representable if it is isomorphic to a small coproduct of representables.

\begin{definition}
	Let $Z$ be an object of $\catD$.
\begin{itemize}[label=--]
\item 	We call $Z$ \emph{exponentiating} if $\catD(-\times Y,Z)$ is representable for all $Y$.
\item 	We call $Z$ \emph{multiexponentiating} if, for all $Y$, the presheaf $\catD(-\times Y,Z)$ is familially representable.
\end{itemize} 	
\end{definition}

 When every object of $\catD$ is multiexponentiating (equivalently, for all pairs $Y,Z$), this is the \emph{cartesian multi-closed} condition of Ad\'amek--Rosick\'y~\cite[Definition~2.8]{zbMATH07186728}.

We also isolate a strong ``atomicity'' condition.

\begin{definition}\label{def:coconnected}
Let $\catD$ have small products. An object $Z\in\catD$ is \emph{coconnected} if the representable presheaf
\[
\catD(-,Z):\catD^{op}\to\Set
\]
preserves small coproducts in $\catD^{op}$, equivalently sends small products in $\catD$ to coproducts in $\Set$.
\end{definition}

\begin{lemmma}\label{lem:coindecomp-multiexp}
If $Z\in\catD$ is coconnected, then $Z$ is multiexponentiating.
\end{lemmma}

\begin{proof}
Fix $Y\in\catD$ and write $\terminal$ for the terminal object of $\catD$. Coconnectedness gives, for each $X\in\catD$, a canonical bijection
\[
\catD(X\times Y,Z)\ \cong\ \catD(X,Z)\ \sqcup\ \catD(Y,Z),
\]
natural in $X$, obtained from preservation of binary coproducts by $\catD(-,Z):\catD^{op}\to \Set$.
Viewed as presheaves in $X$, the second summand is the constant presheaf at the set $\catD(Y,Z)$.
Since $\catD(X,\terminal)\cong 1$ for all $X$, that constant presheaf is (canonically) a coproduct of representables:
\[
\catD(Y,Z)\ \cong\ \coprod_{f\in \catD(Y,Z)} \catD(-,\terminal).
\]
Hence
\[
\catD(-\times Y,Z)\ \cong\ \catD(-,Z)\ \sqcup\ \coprod_{f\in \catD(Y,Z)} \catD(-,\terminal),
\]
a coproduct of representables.
\end{proof}

\begin{lemmma}\label{lem:multiexp-closed-products}
Multiexponentiating objects in a category with small products are closed under small products.
\end{lemmma}

\begin{proof}
Let $(Z_k)_{k\in K}$ be multiexponentiating and fix $Y$. Choose presentations
\[
\catD(-\times Y,Z_k)\ \cong\ \coprod_{i\in I_k}\catD(-,R_{k i})
\qquad(k\in K)
\]
as coproducts of representables.
Then, using that products of representables are representable and that products distribute over coproducts in $\Set$,
\begin{align*}
\catD\!\left(-\times Y, \prod_{k\in K} Z_k\right)
&\cong \prod_{k\in K}\catD(-\times Y,Z_k)\\
&\cong \prod_{k\in K}\coprod_{i\in I_k}\catD(-,R_{k i})\\
&\cong \coprod_{(i_k)_{k\in K}\in \prod_{k\in K}I_k}\ \catD\!\left(-,\prod_{k\in K}R_{k,i_k}\right),
\end{align*}
which is again a small coproduct of representables.
\end{proof}

\subsubsection*{From multiexponentials in $\catD$ to exponentials in $\Fam(\catD)$}

The point of multiexponentiation is that $\Fam(\catD)$ turns coproducts of representables into representables.

\begin{lemmma}\label{lem:Fam-turns-multiexp-into-exp}
Let $\catD$ be a category with binary products and let $\eta:\catD\to\Fam(\catD)$ be the unit. For $Y,Z\in\catD$, the following are equivalent:
\begin{enumerate}
\item\label{it:multiexp}
there exist a set $K$ and objects $(R_k)_{k\in K}$ of $\catD$ with a natural isomorphism
\[
\catD(-\times Y,Z)\ \cong\ \coprod_{k\in K}\catD(-,R_k)
\qquad\text{in }[\catD^{op},\Set];
\]
\item\label{it:singleton-exp}
the exponential $\eta(Y)\Rightarrow \eta(Z)$ exists in $\Fam(\catD)$ and is represented by the family
\[
\sumfam{R_k}{k\in K}.
\]
\end{enumerate}
\end{lemmma}

\begin{proof}
\emph{\ref{it:multiexp}$\Rightarrow$\ref{it:singleton-exp}.}
Let $X=\sumfam{X_i}{i\in I}$ be an object of $\Fam(\catD)$. Since $\eta(Y)$ has singleton index set,
\[
X\times \eta(Y)\ \cong\ \sumfam{X_i\times Y}{i\in I}.
\]
Hence
\[
\Fam(\catD)\bigl(X\times \eta(Y),\,\eta(Z)\bigr)
\ \cong\
\prod_{i\in I}\catD(X_i\times Y,Z).
\]
Using the chosen presentation and distributing products over coproducts in $\Set$,
\begin{align*}
\prod_{i\in I}\catD(X_i\times Y,Z)
&\cong \prod_{i\in I}\ \coprod_{k\in K}\catD(X_i,R_k)\\
&\cong \coprod_{f:I\to K}\ \prod_{i\in I}\catD\bigl(X_i,R_{f(i)}\bigr).
\end{align*}
The right-hand side is, by definition of morphisms in $\Fam(\catD)$,
\[
\Fam(\catD)\Bigl(\sumfam{X_i}{i\in I},\,\sumfam{R_k}{k\in K}\Bigr).
\]
The composite bijection is natural in $X$, so $\sumfam{R_k}{k\in K}$ represents $\eta(Y)\Rightarrow\eta(Z)$.

\smallskip
\emph{\ref{it:singleton-exp}$\Rightarrow$\ref{it:multiexp}.}
Assume $\eta(Y)\Rightarrow \eta(Z)$ exists in $\Fam(\catD)$ and is represented by $E=\sumfam{R_k}{k\in K}$.
For any $X\in\catD$, using full faithfulness of $\eta$ and that $\eta$ preserves products,
\begin{align*}
\catD(X\times Y,Z)
&\cong \Fam(\catD)\bigl(\eta(X\times Y),\,\eta(Z)\bigr)\\
&\cong \Fam(\catD)\bigl(\eta(X)\times \eta(Y),\,\eta(Z)\bigr)\\
&\cong \Fam(\catD)\bigl(\eta(X),\,E\bigr)\\
&\cong \coprod_{k\in K}\catD(X,R_k).
\end{align*}
Naturality in $X$ gives the desired coproduct-of-representables presentation.
\end{proof}

The following criterion is essentially due to Ad\'amek--Rosick\'y~\cite[Theorem~2.11]{zbMATH07186728}, where it is phrased in terms of cartesian multi-closedness. We include the proof, together with the explicit exponential formula, for completeness.

\begin{corollary}[Cartesian closedness criterion for $\Fam$]\label{cor:Fam-ccc-criterion}
	Let $\catD$ be a category with small products. The following are equivalent:
	\begin{enumerate}
		\item $\Fam(\catD)$ is cartesian closed;
		\item every object of $\catD$ is multiexponentiating, i.e.\ for all $A,B\in\catD$ the presheaf $\catD(-\times A,B)$ is a small coproduct of representables.
	\end{enumerate}
	Moreover, assuming these conditions and choosing for each pair $(i,j)$ a presentation
	\[
	\catD(-\times A_i,B_j)\ \cong\ \coprod_{k\in K_{ij}}\catD(-,R_{ij,k}),
	\]
	the exponential of $A=\sumfam{A_i}{i\in I}$ and $B=\sumfam{B_j}{j\in J}$ in $\Fam(\catD)$ is represented by
	\[
	A\Rightarrow B\ \cong\
	\sumfam{
		\prod_{i\in I} R_{i,\,f(i),\,k_i}
	}{
		(f,(k_i))\in
		\coprod_{f\in \Set(I,J)}\ \prod_{i\in I} K_{i,\,f(i)}
	},
	\]
	where $R_{i,f(i),k_i}$ abbreviates $R_{\,i,\,f(i),\,k_i}$.
\end{corollary}

\begin{proof}
Assume every object of $\catD$ is multiexponentiating and fix presentations as in the statement. Define
\[
E\defeq
\sumfam{
\prod_{i\in I} R_{i,\,f(i),\,k_i}
}{
(f,(k_i))\in
\coprod_{f\in \Set(I,J)}\ \prod_{i\in I} K_{i,\,f(i)}
}.
\]
Let $X=\sumfam{X_\ell}{\ell\in L}$ be any object of $\Fam(\catD)$. Using the product formula in $\Fam(\catD)$,
\[X\times A\cong \sumfam{X_\ell\times A_i}{(\ell,i)\in L\times I},\] so
\begin{align*}
\Fam(\catD)(X\times A,B)
&\cong
\coprod_{g\in \Set(L\times I,J)}\ \prod_{(\ell,i)\in L\times I}\catD\bigl(X_\ell\times A_i,\ B_{g(\ell,i)}\bigr)\\
&\cong
\coprod_{(f_\ell)_{\ell\in L}\in \prod_{\ell\in L}\Set(I,J)}\ \prod_{\ell\in L}\ \prod_{i\in I}\catD\bigl(X_\ell\times A_i,\ B_{f_\ell(i)}\bigr).
\end{align*}
For fixed $\ell$, apply the chosen presentation and distribute products over coproducts in $\Set$:
\begin{align*}
\prod_{i\in I}\catD\bigl(X_\ell\times A_i,\ B_{f_\ell(i)}\bigr)
&\cong
\prod_{i\in I}\ \coprod_{k\in K_{i,f_\ell(i)}}\catD\bigl(X_\ell,\ R_{i,f_\ell(i),k}\bigr)\\
&\cong
\coprod_{(k_i)_{i\in I}\in \prod_{i\in I}K_{i,f_\ell(i)}}\ \catD\Bigl(X_\ell,\ \prod_{i\in I}R_{i,f_\ell(i),k_i}\Bigr).
\end{align*}
Substituting back and distributing the remaining product over coproducts in $\Set$ yields a natural bijection
\[
\Fam(\catD)(X\times A,B)\ \cong\ \Fam(\catD)(X,E),
\]
so $E$ represents $A\Rightarrow B$ and $\Fam(\catD)$ is cartesian closed.

Conversely, if $\Fam(\catD)$ is cartesian closed and $A,B\in\catD$, then $\eta(A)\Rightarrow \eta(B)$ exists in $\Fam(\catD)$; applying Lemma~\ref{lem:Fam-turns-multiexp-into-exp} (\ref{it:singleton-exp}$\Rightarrow$\ref{it:multiexp}) gives a coproduct-of-representables presentation of $\catD(-\times A,B)$. Hence every object of $\catD$ is multiexponentiating.
\end{proof}

\begin{remark}[Multiexponentiation via generic maps]\label{rem:diers}
Let $\catD$ have small products and fix $A,B\in\catD$. To give a natural isomorphism
\[
\catD(-\times A,B)\ \cong\ \coprod_{k\in K}\catD(-,R_k)
\]
is equivalent to giving a family of arrows $(\rho_k:R_k\times A\to B)_{k\in K}$ such that every $f:X\times A\to B$
factorises uniquely as
\[
f=\rho_k\circ(u\times \id_A)
\qquad
\text{for a unique }k\in K\text{ and a unique }u:X\to R_k.
\]
Finally, since multiexponentiating objects are closed under small products (Lemma~\ref{lem:multiexp-closed-products}),
it suffices to verify multiexponentiation against a set of generators under products.
\end{remark}

\subsection{Cartesian closedness of \texorpdfstring{$\Dist(\catC)$}{Dist(C)}}

We now apply the criterion above with $\catD\defeq \Fam_{\Pi}(\catC)=\Fam(\catC^{op})^{op}$. The crucial observation is that the image of $\catC$ inside $\catD$ consists of coconnected objects.

\begin{lemmma}\label{lem:base-objects-coindecomp}
For each $Z\in\catC$, its image $\widehat{Z}\in\catD=\Fam(\catC^{op})^{op}$ is coconnected.
\end{lemmma}

\begin{proof}
Let $\prodfam{X_i}{i\in I}$ be an object of $\catD$. A morphism in $\catD$ from $\prodfam{X_i}{i\in I}$ to $\widehat{Z}$ is, by definition of $\Fam_{\Pi}$, precisely a choice of an index $i\in I$ together with an arrow $X_i\to Z$ in $\catC$. Hence
\[
\catD\!\left(\prodfam{X_i}{i\in I},\widehat{Z}\right)
\ \cong\
\coprod_{i\in I}\catC(X_i,Z),
\]
naturally in the family $(X_i)_{i\in I}$. This is exactly preservation of products by $\catD(-,\widehat{Z})$.
\end{proof}

By Lemma~\ref{lem:coindecomp-multiexp}, every $\widehat{Z}$ is multiexponentiating; by Lemma~\ref{lem:multiexp-closed-products}, every object of $\catD$, being a product of such $\widehat{Z}$'s, is multiexponentiating. Corollary~\ref{cor:Fam-ccc-criterion} therefore implies that $\Fam(\catD)=\Dist(\catC)$ is cartesian closed. Unwinding the presentations yields the explicit formula below.

\begin{theorem}\label{thm:ccc}
The category $\Dist(\catC)$ is cartesian closed.

More explicitly, denote objects of $\Dist(\catC)$ by formal coproducts of formal products:
\[
A=\sum_{j\in J}\prod_{i\in I_j} C_{j i},
\qquad
B=\sum_{j'\in J'}\prod_{i'\in I'_{j'}} C'_{j' i'}.
\]
For each $j\in J$, $j'\in J'$ and $i'\in I'_{j'}$, define the set of \emph{partial maps}
\[
\Par(j;\,j',i')\ \defeq\
\{\ast\}\ \sqcup\ \coprod_{i\in I_j} \catC\!\left(C_{j i},\,C'_{j' i'}\right).
\]
For $r\in \Par(j;\,j',i')$, define an object $E^{j',i'}(r)\in \Fam_{\Pi}(\catC)$ by
\[
E^{j',i'}(r)\defeq
\begin{cases}
\widehat{C'_{j' i'}} & r=\ast,\\[0.2em]
\terminal_{\Fam_{\Pi}(\catC)} & r\neq \ast.
\end{cases}
\]
Then the exponential $A\Rightarrow B$ is represented by the family
\begin{equation}\label{eq:exponential-formula}
\sumfam{
\prod_{(j,i')\in \coprod_{j\in J}I'_{g(j)}} E^{g(j),i'}(r_{j,i'})
}{
\bigl(g,(r_{j,i'})\bigr)\in
\coprod_{g\in \Set(J,J')}\ \prod_{j\in J}\ \prod_{i'\in I'_{g(j)}} \Par\!\left(j;\,g(j),i'\right)
},
\end{equation}
where the product is taken in $\Fam_{\Pi}(\catC)$.
\end{theorem}

\begin{proof}
Let $\catD=\Fam_{\Pi}(\catC)$, so $\Dist(\catC)=\Fam(\catD)$.

\smallskip
\noindent\emph{Step 1: presentations in $\catD$.}
Fix $j\in J$, $j'\in J'$ and $i'\in I'_{j'}$. The object $\widehat{C'_{j'i'}}\in\catD$ is coconnected by Lemma~\ref{lem:base-objects-coindecomp}, hence multiexponentiating by Lemma~\ref{lem:coindecomp-multiexp}. Applying the proof of Lemma~\ref{lem:coindecomp-multiexp} with $Y=\prod_{i\in I_j}\widehat{C_{j i}}$ and $Z=\widehat{C'_{j' i'}}$ yields a natural isomorphism of presheaves
\[
\catD\!\left(-\times \prod_{i\in I_j}\widehat{C_{j i}},\,\widehat{C'_{j' i'}}\right)
\ \cong\
\coprod_{r\in \Par(j;\,j',i')}\ \catD\bigl(-,\,E^{j',i'}(r)\bigr),
\]
where the indexing set is
\[
\{\ast\}\ \sqcup\
\catD\!\left(\prod_{i\in I_j}\widehat{C_{j i}},\,\widehat{C'_{j' i'}}\right)
\ \cong\
\{\ast\}\ \sqcup\
\coprod_{i\in I_j}\catC(C_{j i},C'_{j' i'})
\]
by the hom-formula in $\Fam_{\Pi}(\catC)$.
This matches the definition of $\Par(j;\,j',i')$ and $E^{j',i'}$ above: the $\ast$ summand corresponds to the representable $\catD(-,\widehat{C'_{j' i'}})$, while the remaining summands are represented by the terminal object of $\catD$.

\smallskip
\noindent\emph{Step 2: products in the codomain.}
Since multiexponentiating objects are closed under products (Lemma~\ref{lem:multiexp-closed-products}), the object
\[
\prod_{i'\in I'_{j'}}\widehat{C'_{j'i'}}\ \in\catD
\]
is multiexponentiating. Expanding the product of coproducts in $\Set$ yields a presentation
\[
\catD\!\left(-\times \prod_{i\in I_j}\widehat{C_{j i}},\,\prod_{i'\in I'_{j'}}\widehat{C'_{j' i'}}\right)
\ \cong\
\coprod_{(r_{i'})_{i'\in I'_{j'}}\in \prod_{i'\in I'_{j'}}\Par(j;\,j',i')}
\catD\!\left(-,\prod_{i'\in I'_{j'}}E^{j',i'}(r_{i'})\right),
\]
with products on the right taken in $\catD=\Fam_{\Pi}(\catC)$.

\smallskip
\noindent\emph{Step 3: passage to $\Fam(\catD)$.}
Now apply Corollary~\ref{cor:Fam-ccc-criterion} to $\Fam(\catD)=\Dist(\catC)$, with
\[
A=\sum_{j\in J}A_j,\quad A_j=\prod_{i\in I_j}\widehat{C_{j i}},
\qquad
B=\sum_{j'\in J'}B_{j'},\quad B_{j'}=\prod_{i'\in I'_{j'}}\widehat{C'_{j' i'}}.
\]
The corollary expresses $A\Rightarrow B$ as a coproduct indexed by a choice function $g:J\to J'$ together with, for each $j$, a choice of an index in the multiexponentiation presentation of $\catD(-\times A_j,B_{g(j)})$. Substituting the presentations from Steps~1--2 and rewriting the resulting indices gives exactly the formula~\eqref{eq:exponential-formula}.
\end{proof}

\begin{remark}[Why ``partial maps''?]\label{rem:partial-maps}
The indexing in~\eqref{eq:exponential-formula} arises by repeatedly distributing products over coproducts in $\Set$.
The function $g:J\to J'$ records, for each coproduct summand of $A$, which coproduct summand of $B$ it targets.
For fixed $j$ and $i'\in I'_{g(j)}$, the choice
\[
r_{j,i'}\in \Par(j;\,g(j),i')
=\{\ast\}\ \sqcup\ \coprod_{i\in I_j}\catC(C_{j i},C'_{g(j),i'})
\]
records whether the $i'$-th coordinate is ``left free'' (the case $r_{j,i'}=\ast$) or is supplied by a specific arrow out of one of the factors of the product $\prod_{i\in I_j}C_{j i}$ (the case $r_{j,i'}\neq \ast$). In the former case, $\widehat{C'_{g(j),i'}}$ remains as a factor of the representing product; in the latter case, that factor is replaced by the terminal object of $\Fam_{\Pi}(\catC)$.
\end{remark}

\begin{remark}[Inductive reconstruction of exponentials]\label{rem:Inductive-definition}
Every object of $\Dist(\catC)$ is canonically a coproduct of formal products of objects of $\catC$. Exponentials can therefore be reconstructed from the following identities:
\begin{itemize}
\item If $A$ lies in the image of $\Fam_{\Pi}(\catC)\to\Dist(\catC)$ and $B$ lies in the image of $\catC\to\Dist(\catC)$, then
\[
A\Rightarrow B \;\cong\; B\ \sqcup\ \coprod_{f\in \Dist(\catC)(A,B)} \terminal,
\]
the special case of Theorem~\ref{thm:ccc} with one target coordinate.
\item If $A$ lies in the image of $\Fam_{\Pi}(\catC)\to\Dist(\catC)$ and $(B_k)_{k\in K}$ is any family in $\Dist(\catC)$, then
\[
A\Rightarrow \Bigl(\coprod_{k\in K} B_k\Bigr) \;\cong\; \coprod_{k\in K} (A\Rightarrow B_k).
\]
\item For any family $(A_j)_{j\in J}$,
\[
\Bigl(\coprod_{j\in J} A_j\Bigr)\Rightarrow B \;\cong\; \prod_{j\in J} (A_j\Rightarrow B).
\]
\item For any family $(B_i)_{i\in I}$,
\[
A\Rightarrow \Bigl(\prod_{i\in I} B_i\Bigr) \;\cong\; \prod_{i\in I} (A\Rightarrow B_i).
\]
\end{itemize}
Unfolding these rules yields the explicit partial-map indexing of Theorem~\ref{thm:ccc}.
\end{remark}

\begin{remark}[Comparison with lattice-theoretic complete distributivity]
Completely distributive lattices are complete Heyting algebras, with exponentials $a\Rightarrow b=\bigvee\{a'\mid a\wedge a'\le b\}$. Theorem~\ref{thm:ccc} is qualitatively different: although $\Dist(\catC)$ is cartesian closed, doubly-infinitary distributive categories need not admit exponentials (Section~\ref{sec:examples}). Moreover, completely distributive lattices of the form $\Dist(\catC)$ are necessarily trivial, so this phenomenon is genuinely non-thin. Cartesian closedness (and even Grothendieck topos structure) can nevertheless be recovered under stronger distributivity hypotheses together with generator conditions; see~\cite{marmolejo2012completely}.
\end{remark}

\begin{remark}[Exponentials in free infinitary distributive categories]
The same multiexponentiation mechanism yields exponentials in free infinitary distributive categories in the expected restricted cases; for instance, in $\Fam(\FinFam(\catC^{op})^{op})$ exponentials exist whenever the exponent lies in the image of the corresponding free product completion.
\end{remark}

\section{Pseudomonadic structure and doubly-infinitary distributivity}\label{sec:pseudomonad-Dist}

This section gives the structural explanation for the notation $\Dist$. We first define \emph{doubly-infinitary distributive categories}. Although the condition is natural and fundamental, it does not seem to have been systematically isolated in the literature. Thus, beyond the main cartesian-closedness theorem for free doubly-infinitary distributive categories, part of the contribution of this paper is to bring this categorical structure into focus and to clarify its basic role.

After the basic definition, we recall the canonical pseudodistributive law
\begin{equation}\label{eq:pseudodistributive-law}
\Fam_{\Pi}\circ\Fam_{\Sigma}\Longrightarrow \Fam_{\Sigma}\circ\Fam_{\Pi},
\end{equation} 
which induces the composite pseudomonad $\Dist=\Fam_{\Sigma}\circ\Fam_{\Pi}$, and prove that its pseudoalgebras are precisely doubly-infinitary distributive categories. This section is not needed for the proof of Theorem~\ref{thm:ccc}.

\subsection{Doubly-infinitary distributive categories}

We now give the explicit definition of the \textit{doubly-infinitary distributivity} condition encoded by the composite pseudomonad. In the form stated below, this condition is precisely the pseudoalgebra condition for the composite pseudomonad obtained from \eqref{eq:pseudodistributive-law}; see Lemma~\ref{lem:dist-are-pseudoalgebras}.

\begin{definition}[Doubly-infinitary distributive category]\label{def:doubly-infinitary-distributivity}
Let $\catC$ be a category with small products and small coproducts. For any family of families $(C_{j i})_{j\in J,\ i\in I_j}$, there is a canonical morphism
\begin{equation}\label{eq:canonical-morphism-that-defines}
[\langle\iota_{r_j}\circ \pi_j\mid j\in J\rangle\mid r\in \prod_{j\in J} I_j]:
\left(\coprod_{r\in  \prod\limits_{j\in J} I_j} \prod_{j\in J} C_{\,j,\,r_j}\right)
\longrightarrow
\left(\prod_{j\in J} \coprod_{i\in I_j} C_{\,j i}\right),
\end{equation}
where $\pi_j$ are the product projections and $\iota_{r_j}$ the coproduct coprojections. We say that $\catC$ is \emph{doubly-infinitary distributive} if~\eqref{eq:canonical-morphism-that-defines} is invertible for every such family.
\end{definition}

\begin{remark}[Functorial formulation]\label{rem:concise-definition}
As for ordinary distributive categories, Definition~\ref{def:doubly-infinitary-distributivity} admits a concise reformulation. Let
\[
\coprod:\Fam(\catC)\to\catC
\]
be the coproduct functor, left adjoint to the unit $\eta^\Sigma_\catC$. Then $\catC$ is doubly-infinitary distributive if and only if it has small products and small coproducts and the functor
\begin{equation}\label{eq:coproduct-functor}
\coprod:\Fam(\catC)\to\catC
\end{equation}
preserves small products.
\end{remark}

\subsection{The distributive law}

The pseudomonads $\Fam_{\Pi}=\Fam(-^{op})^{op}$ and $\Fam_{\Sigma}=\Fam(-)$ on $\Cat$ admit a canonical pseudodistributive law
\[
\lambda:\ \Fam_{\Pi}\circ \Fam_{\Sigma}\ \Longrightarrow\ \Fam_{\Sigma}\circ \Fam_{\Pi},
\]
in the sense of~\cite[Definition~11.4]{marmolejo1999distributive}; see also~\cite{zbMATH02116176, zbMATH05256222, MR3491845, zbMATH06970806} for generalities. Concretely, for a category $\catC$, the component
\[
\lambda_\catC:\ \Fam\!\bigl(\Fam(\catC)^{op}\bigr)^{op}\ \longrightarrow\ \Fam\!\bigl(\Fam(\catC^{op})^{op}\bigr)
\]
acts on objects by distributing products over coproducts:
\[
\lambda_\catC\!\left(\prodfam{\sumfam{C_{j,i}}{i\in I_j}}{j\in J}\right)
\ =\
\sumfam{\prodfam{C_{j,r_j}}{j\in J}}{r\in \prod_{j\in J} I_j}.
\]
On morphisms one proceeds similarly by universal properties; the construction is determined uniquely up to coherent isomorphism by the free (co)product completions. Following~\cite[Theorem~7.1]{von2018polynomials} (with base $\mathbb{B}=\Set$), these data define a pseudodistributive law, which is essentially unique~\cite[Remark~34, Corollary~49]{walker2019distributive}.

Consequently, the composite $\Dist=\Fam_{\Sigma}\circ\Fam_{\Pi}$ inherits a pseudomonad structure on $\Cat$.

\subsection{Pseudoalgebras and doubly-infinitary distributivity}

The compatibility encoded by $\lambda$ is precisely the condition of Definition~\ref{def:doubly-infinitary-distributivity}. In particular, pseudoalgebras for $\Dist$ are exactly doubly-infinitary distributive categories.

\begin{lemmma}[$\Dist$-pseudoalgebras]\label{lem:dist-are-pseudoalgebras}
The $2$-category $\DistCat$ of $\Dist$-pseudoalgebras consists precisely of doubly-infinitary distributive categories, product- and coproduct-preserving functors, and natural transformations.
\end{lemmma}

\begin{proof}
By construction, $\Dist$ is the composite pseudomonad $\Fam_{\Sigma}\circ\Fam_{\Pi}$ obtained from the pseudodistributive law $\lambda$. Thus a $\Dist$-pseudoalgebra is a category $\catC$ equipped with:
\begin{itemize}
\item a $\Fam_{\Pi}$-pseudoalgebra structure, i.e.\ small products;
\item a $\Fam_{\Sigma}$-pseudoalgebra structure, i.e.\ small coproducts;
\end{itemize}
together with the compatibility condition imposed by $\lambda$, namely that the coproduct structure is a morphism in $\ProdCat$. Concretely, this means that the coproduct functor~\eqref{eq:coproduct-functor} preserves products. Unwinding the product formula in $\Fam(\catC)$ shows that this is equivalent to invertibility of the canonical comparisons~\eqref{eq:canonical-morphism-that-defines}. This is exactly Definition~\ref{def:doubly-infinitary-distributivity}.
\end{proof}

In particular, $\Dist(\catC)$ is the \emph{free} doubly-infinitary distributive category on $\catC$: for any functor $F:\catC\to\catD$ into a doubly-infinitary distributive category $\catD$, there exists an essentially unique extension
\[
\widehat{F}:\Dist(\catC)\to\catD
\]
preserving both products and coproducts.

\begin{remark}
A convenient way to verify existence of $\lambda$ is to view it as a lifting of the free coproduct completion pseudomonad along the pseudomonadic forgetful $2$-functor $\ProdCat\to\Cat$ in the sense of~\cite{zbMATH02116176}. One obtains pseudomonadic biadjunctions
\begin{center}
\begin{tikzcd}
\Cat \arrow[rr, bend left] & \bot (\epsilon,\eta) & \ProdCat \arrow[ll, bend left]\arrow[rr, bend left, "Fam(-)"]
& \bot (\epsilon ', \eta ')
& \DistCat \arrow[ll, bend left]
\end{tikzcd}
\end{center}
and $\lambda$ is the coherence data making the lift well-defined.
\end{remark}
\section{Free doubly-infinitary distributive categories on coproducts of categories}\label{sec:discrete-case}

This section has two purposes. First, we compute explicitly how the free doubly-infinitary distributive completion behaves on coproducts of categories, and then specialize this computation to discrete categories. This gives a concrete family of examples and makes transparent the effect of the composite completion $\Dist$ on a set of generators. Second, with the presentation used here, the same calculation brings into view a bicategorical biproduct phenomenon in the $2$-category of categories with products. This latter viewpoint is not needed for our development, but it clarifies how the free product-completion part of $\Dist$ behaves on bicategorical coproducts of product-complete categories, and points to a broader form of bicategorical semi-additivity.

\subsection{A biproduct computation in \texorpdfstring{$\ProdCat$}{ProdCat}}

We use bicategorical limits and colimits throughout this subsection; see~\cite{zbMATH03680046, zbMATH04008629, MR3491845, zbMATH06881682}. The main observation is a bicategorical biproduct phenomenon: in the $2$-category $\ProdCat$ of categories with small products, product-preserving functors, and natural transformations, small bicategorical products also serve as bicategorical coproducts. Thus $\ProdCat$ has small bicategorical biproducts; in analogy with the one-dimensional case, one may say that $\ProdCat$ is \emph{infinitary bicategorically semi-additive}.

\begin{proposition}\label{prop:ProdCat-biproduct}
	Let $(A_i)_{i\in L}$ be a small family of categories with small products. Then the cartesian product
	\[
	\prod_{i\in L}A_i
	\]
	is both a bicategorical product and a bicategorical coproduct of the family $(A_i)_{i\in L}$ in $\ProdCat$.
\end{proposition}

\begin{proof}
	The forgetful $2$-functor $\ProdCat\to\Cat$ is pseudomonadic and hence creates bicategorical products. Therefore the cartesian product $\prod_{i\in L}A_i$, with products computed componentwise, is the bicategorical product of the family $(A_i)_{i\in L}$ in $\ProdCat$.
	
	It remains to describe the bicategorical coproduct structure. For each $t\in L$, define a product-preserving functor
	\[
	\kappa_t:A_t\longrightarrow \prod_{i\in L}A_i
	\]
	by setting
	\[
	(\kappa_t(x))_i=
	\begin{cases}
		x & \text{if } i=t,\\
		\terminal_{A_i} & \text{if } i\neq t,
	\end{cases}
	\]
	where $\terminal_{A_i}$ denotes the terminal object of $A_i$. These functors preserve products because products in $\prod_{i\in L}A_i$ are computed componentwise.
	
	For every category with small products $B$, the functors $\kappa_t$ induce a functor
	\begin{equation}\label{eq:equivalence-bicategorical-products}
		\ProdCat\!\left(\prod_{i\in L}A_i,B\right)
		\longrightarrow
		\prod_{i\in L}\ProdCat(A_i,B),
		\qquad
		H\longmapsto (H\circ\kappa_i)_{i\in L}.
	\end{equation}
	This functor is an equivalence. Indeed, given a family of product-preserving functors
	\[
	F_i:A_i\to B
	\qquad (i\in L),
	\]
	define
	\[
	F:\prod_{i\in L}A_i\longrightarrow B,
	\qquad
	F((a_i)_{i\in L})\defeq \prod_{i\in L}F_i(a_i).
	\]
	This functor preserves products, since each $F_i$ does and iterated products in $B$ are canonically isomorphic. Moreover,
	\[
	F\circ \kappa_t\cong F_t
	\]
	naturally in $t$. Conversely, if $H:\prod_{i\in L}A_i\to B$ preserves products, then every object $(a_i)_{i\in L}$ decomposes canonically as
	\[
	(a_i)_{i\in L}\cong \prod_{i\in L}\kappa_i(a_i)
	\]
	in $\prod_{i\in L}A_i$, and hence
	\[
	H((a_i)_{i\in L})
	\cong
	\prod_{i\in L}H(\kappa_i(a_i)).
	\]
	Thus $H$ is recovered, up to canonical isomorphism, from the family $(H\circ \kappa_i)_{i\in L}$. The same argument applies to natural transformations. Therefore~\eqref{eq:equivalence-bicategorical-products} is an equivalence of categories, and $\prod_{i\in L}A_i$ is also a bicategorical coproduct.
\end{proof}

Thus $\ProdCat$ has \emph{small bicategorical biproducts}: its bicategorical products and coproducts exist and agree up to equivalence. The finite version is obtained by the same argument. That is to say, the $2$-category $\mathbf{Fin}\ProdCat$ of categories with finite products, finite-product-preserving functors, and natural transformations has finite bicategorical biproducts.

The bicategorical biproducts described above are part of a more general phenomenon in two-dimensional category theory, although this perspective does not seem to have been emphasized systematically in the literature. In analogy with the one-dimensional relationship between biproducts and enrichment over commutative monoids, bicategorical semi-additivity is related to enrichment over suitable $2$-categories of symmetric monoidal categories. In the infinitary case, the relevant enrichment should involve symmetric infinitary monoidal categories. From this point of view, the existence of bicategorical products or coproducts, together with the appropriate enriched structure, should force the existence of suitable bicategorical biproducts. We do not develop this general theory here, and leave these observations for future work.

For the present paper, we only need the concrete computation above: the $2$-category $\ProdCat$ has small bicategorical biproducts. The same argument gives finite bicategorical biproducts in the $2$-category $\mathbf{Fin}\ProdCat$ of categories with finite products, finite-product-preserving functors, and natural transformations.

\subsection{The construction}

We now apply the preceding biproduct computation to the composite completion $\Dist=\Fam_{\Sigma}\circ\Fam_{\Pi}$. The point is that $\Fam_{\Pi}:\Cat\to\ProdCat$ is the free product-completion pseudofunctor, hence is left biadjoint to the forgetful $2$-functor $\ProdCat\to\Cat$. It therefore preserves bicategorical coproducts. Since bicategorical coproducts in $\ProdCat$ are represented by cartesian products, Proposition~\ref{prop:ProdCat-biproduct} gives the following description.

\begin{theorem}\label{theo:on-coproduct-free}
	Let $(\catC_\ell)_{\ell\in L}$ be a small family of categories. Then
	\[
	\Dist\!\left(\coprod_{\ell\in L}\catC_\ell\right)
	\simeq
	\Fam_{\Sigma}\!\left(\prod_{\ell\in L}\Fam_{\Pi}(\catC_\ell)\right).
	\]
	In particular, for two categories $\catA$ and $\catB$,
	\[
	\Dist(\catA\sqcup\catB)
	\simeq
	\Fam_{\Sigma}\!\bigl(\Fam_{\Pi}(\catA)\times \Fam_{\Pi}(\catB)\bigr).
	\]
\end{theorem}

\begin{proof}
	Since $\Fam_{\Pi}:\Cat\to\ProdCat$ is left biadjoint to the forgetful $2$-functor $\ProdCat\to\Cat$, it preserves bicategorical coproducts. Therefore
	\[
	\Fam_{\Pi}\!\left(\coprod_{\ell\in L}\catC_\ell\right)
	\simeq
	\coprod_{\ell\in L}\Fam_{\Pi}(\catC_\ell)
	\]
	in $\ProdCat$. By Proposition~\ref{prop:ProdCat-biproduct}, the bicategorical coproduct on the right is represented by the cartesian product
	\[
	\prod_{\ell\in L}\Fam_{\Pi}(\catC_\ell).
	\]
	Hence
	\[
	\Fam_{\Pi}\!\left(\coprod_{\ell\in L}\catC_\ell\right)
	\simeq
	\prod_{\ell\in L}\Fam_{\Pi}(\catC_\ell)
	\]
	as objects of $\ProdCat$. Forgetting to $\Cat$ and applying $\Fam_{\Sigma}$ gives
	\[
	\Dist\!\left(\coprod_{\ell\in L}\catC_\ell\right)
	=
	\Fam_{\Sigma}\!\left(\Fam_{\Pi}\!\left(\coprod_{\ell\in L}\catC_\ell\right)\right)
	\simeq
	\Fam_{\Sigma}\!\left(\prod_{\ell\in L}\Fam_{\Pi}(\catC_\ell)\right),
	\]
	as required. The binary case is the special case $L=\{0,1\}$.
\end{proof}

Using the result above, we now specialize the computation to discrete categories.

\begin{corollary}\label{theo:on-discrete-set-free}
	Let $A$ be a set, regarded as a discrete category. Then the free doubly-infinitary distributive category on $A$ is equivalent to
	\[
	\Dist(A)\simeq \Fam\!\left(\prod_{a\in A}\Set^{op}\right),
	\]
	where $\prod_{a\in A}\Set^{op}$ denotes the cartesian product of the $A$-indexed family of categories $(\Set^{op})_{a\in A}$.
\end{corollary}

\begin{proof}
	The free product completion of the terminal category is equivalent to $\Set^{op}$:
	\[
	\Fam_{\Pi}(\terminal)\simeq \Set^{op}.
	\]
	Since $A$ is a discrete category, it is the coproduct in $\Cat$ of $A$ copies of the terminal category:
	\[
	A\cong \coprod_{a\in A}\terminal.
	\]
	Applying Theorem~\ref{theo:on-coproduct-free}, we get
	\[
	\Dist(A)
	\simeq
	\Fam_{\Sigma}\!\left(\prod_{a\in A}\Fam_{\Pi}(\terminal)\right)
	\simeq
	\Fam_{\Sigma}\!\left(\prod_{a\in A}\Set^{op}\right).
	\]
	Since $\Fam_{\Sigma}=\Fam$, this is precisely
	\[
	\Dist(A)
	\simeq
	\Fam\!\left(\prod_{a\in A}\Set^{op}\right).
	\]
\end{proof}

We also spell out a direct proof, independent of the bicategorical biproduct computation. Since $A$ is discrete, an object of $\Fam(A)$ is precisely a family of elements of $A$ indexed by some set, equivalently a function into $A$. Thus
\[
\Fam(A)\simeq \Set/A.
\]
Taking fibres gives the standard equivalence
\[
\Set/A\simeq \prod_{a\in A}\Set.
\]
Since $A$ is discrete, $A^{op}=A$. Consequently,
\[
\Fam(A^{op})^{op}
\simeq
\Fam(A)^{op}
\simeq
(\Set/A)^{op}
\simeq
\left(\prod_{a\in A}\Set\right)^{op}
\simeq
\prod_{a\in A}\Set^{op}.
\]
Applying $\Fam(-)$ then gives
\[
\Dist(A)
=
\Fam\!\bigl(\Fam(A^{op})^{op}\bigr)
\simeq
\Fam\!\left(\prod_{a\in A}\Set^{op}\right),
\]
as before.

\section{Examples}\label{sec:examples}

This section collects examples and counterexamples separating doubly-infinitary distributivity from related conditions. Recall that every doubly-infinitary distributive category is infinitary distributive, and that every completely distributive category is doubly-infinitary distributive~\cite{marmolejo2012completely}.

\subsection{The \texorpdfstring{$\Fam$}{Fam} construction}

For any category $\catC$, the free coproduct completion $\Fam(\catC)$ is extensive, and it is lextensive whenever $\catC$ has finite limits. However, $\Fam(\catC)$ need not have small products, and hence need not be doubly-infinitary distributive. If $\catC$ already has small products, then the free coproduct completion takes place inside $\ProdCat$, and the resulting category is doubly-infinitary distributive.

\begin{lemmma}\label{ex:famprod}
	If $\catC$ has small products, then $\Fam(\catC)$ is doubly-infinitary distributive.
\end{lemmma}

\begin{proof}
	Regard $\catC$ as an object of $\ProdCat$. By Section~\ref{sec:pseudomonad-Dist}, the lifted free coproduct-completion pseudofunctor
	\[
	\Fam:\ProdCat\to\DistCat
	\]
	is left biadjoint to the forgetful pseudofunctor
	\[
	\DistCat\to\ProdCat.
	\]
	Therefore $\Fam(\catC)$ carries the structure of a doubly-infinitary distributive category.
\end{proof}

\begin{example}
	The category of sets is doubly-infinitary distributive. Indeed,
	\[
	\Set\simeq \Fam(\terminal)\simeq \Dist(\initial),
	\]
	where $\terminal$ is the terminal category and $\initial$ is the initial category. Thus $\Set$ is the free $\Dist$-algebra on $\initial$, and hence the initial object of $\DistCat$.
\end{example}

\begin{example}[Polynomials/containers]
	The category $\Fam(\Set^{op})=\Dist(\terminal)$ is the free doubly-infinitary distributive category on $\terminal$. It is commonly studied as the category of set-based polynomials, or containers~\cite{abbott2003categories}. Theorem~\ref{thm:ccc} recovers its cartesian closedness, as observed in~\cite{altenkirch2010higher}.
\end{example}

The preceding lemma also gives many less free examples. For instance, since $\Top$ has small products, the category $\Fam(\Top)$ is doubly-infinitary distributive. We will use this observation below to obtain examples of doubly-infinitary distributive categories that are not cartesian closed.

\subsection{A transfer lemma}

We record the following elementary transfer principle, which gives a convenient way of producing examples by comparison with categories where products and coproducts are easier to compute.

\begin{lemmma}[Transfer of doubly-infinitary distributivity]\label{lem:conservative-transfer}
	Let $\catC$ and $\catE$ be categories with small products and small coproducts, and let
	\[
	U:\catC\to\catE
	\]
	be a functor that creates small products and small coproducts. If $\catE$ is doubly-infinitary distributive, then so is $\catC$.
\end{lemmma}

\begin{proof}
	Let $(C_{ji})_{j\in J,\ i\in I_j}$ be a family of objects of $\catC$, and let $d_{\catC}$ be the canonical comparison~\eqref{eq:canonical-morphism-that-defines} in $\catC$. Since $U$ creates the relevant products and coproducts, $U(d_{\catC})$ identifies with the corresponding canonical comparison for the family $(U(C_{ji}))_{j\in J,\ i\in I_j}$ in $\catE$. This comparison is invertible because $\catE$ is doubly-infinitary distributive. By creation of the products and coproducts appearing in the comparison, $d_{\catC}$ is invertible. Hence $\catC$ is doubly-infinitary distributive.
\end{proof}

\begin{example}[$M$-sets]
	Let $M$ be a monoid, and denote by $M\text{-}\Set$ the category of left $M$-sets and $M$-equivariant maps. The forgetful functor
	\[
	U:M\text{-}\Set\to\Set
	\]
	creates small products and small coproducts. Products carry the diagonal action,
	\[
	m\cdot (x_i)_{i\in I}\defeq (m\cdot x_i)_{i\in I},
	\]
	and coproducts carry the summandwise action,
	\[
	m\cdot \iota_i(x)\defeq \iota_i(m\cdot x).
	\]
	Since $\Set$ is doubly-infinitary distributive, Lemma~\ref{lem:conservative-transfer} implies that $M\text{-}\Set$ is doubly-infinitary distributive.
\end{example}

\begin{example}[Functor categories and presheaves]\label{ex:presheaves}
	Let $\catD$ be a small category. If $\catC$ is doubly-infinitary distributive, then so is the cartesian product
	\[
	\prod_{d\in \Ob(\catD)}\catC.
	\]
	Indeed, products of $\Dist$-pseudoalgebras are computed componentwise. Moreover, the evaluation functor
	\[
	E:\left[\catD^{op},\catC\right]\to \prod_{d\in \Ob(\catD)}\catC
	\]
	creates small products and small coproducts. Therefore, by Lemma~\ref{lem:conservative-transfer}, the functor category $\left[\catD^{op},\catC\right]$ is doubly-infinitary distributive.
	
	In particular, for any small category $\catD$, the presheaf category $[\catD^{op},\Set]$ is doubly-infinitary distributive. Taking $\catD$ to be the walking graph, the category $\mathbf{Graph}$ of graphs in $\Set$ is doubly-infinitary distributive.
\end{example}

\subsection{Posets}

A poset $\catC$ is a doubly-infinitary distributive category if and only if it is a completely distributive lattice; see, for instance,~\cite{zbMATH04136008}.

Moreover, only the trivial poset is extensive. Indeed, in a poset, parallel arrows are equal, so in the coproduct diagram
\[
a \xrightarrow{\iota_1} a\sqcup a \xleftarrow{\iota_2} a
\]
one has $\iota_1=\iota_2$. Disjointness of coproducts would force the pullback of $\iota_1$ and $\iota_2$ to be $\bot$, but that pullback is $a$ itself. Hence $a=\bot$ for every object $a$.

\medskip
\noindent\emph{Consequently, any non-trivial completely distributive lattice yields a doubly-infinitary distributive category that is not extensive.}

\subsection{Topological spaces}

The category $\Top$ of topological spaces is infinitary distributive in the classical sense: finite products distribute over small coproducts. It is, however, \emph{not} doubly-infinitary distributive.

Consider~\eqref{eq:canonical-morphism-that-defines} with $J=\NN$, $I_j=2$ for every $j\in\NN$, and $C_{ji}=1$, the one-point space, for every $j$ and $i$. This gives a continuous map
\[
d:\ \coprod_{r\in 2^{\NN}}\ \prod_{j\in\NN}1
\longrightarrow
\prod_{j\in\NN}\ \coprod_{i\in 2}1.
\]
Since $\prod_{j\in\NN}1\cong 1$, the domain is the coproduct $\coprod_{r\in 2^{\NN}}1$, namely the discrete space on the underlying set $2^{\NN}$. The codomain is the product $\prod_{j\in\NN}2$, namely the Cantor space $2^{\NN}$ with the product topology, which is compact by Tychonoff's theorem~\cite{zbMATH01002289}. The underlying function of $d$ is the identity on $2^{\NN}$, but $d$ is not an isomorphism in $\Top$: its inverse would be a continuous bijection from a compact space to an infinite discrete space, which is impossible since continuous images of compact spaces are compact and compact subsets of a discrete space are finite.

\medskip
\noindent\emph{Thus $\Top$ is infinitary distributive, indeed lextensive, but not doubly-infinitary distributive.}

\subsubsection{Pointfree topology}

A \emph{frame} is a complete lattice in which finite meets distribute over arbitrary joins; equivalently, it is an infinitary distributive poset when viewed as a category. The opposite category of frames is the category of locales,
\[
\Loc=\Frm^{op};
\]
see~\cite{zbMATH05898723}.

We first recall that $\Loc$ is extensive. This may be verified directly by proving the corresponding coextensivity property of $\Frm$; compare Taylor~\cite[Section~5.5]{taylor1999practical}. Let
\[
p:Z\to X\sqcup Y
\]
be a locale morphism, corresponding to a frame homomorphism
\[
p^*:\Omega(X)\times\Omega(Y)\to\Omega(Z).
\]
The elements $(1,0)$ and $(0,1)$ of the product frame $\Omega(X)\times\Omega(Y)$ are complementary. Hence their images
\[
u\defeq p^*(1,0),
\qquad
v\defeq p^*(0,1)
\]
are complementary elements of $\Omega(Z)$. These complementary opens decompose $Z$ as the coproduct of the corresponding open sublocales,
\[
Z\cong Z_u\sqcup Z_v.
\]
A direct calculation with the defining pullback squares identifies these open sublocales with
\[
Z\times_{X\sqcup Y}X
\qquad\text{and}\qquad
Z\times_{X\sqcup Y}Y.
\]
Thus one obtains the required extensivity isomorphism
\[
Z \cong
\bigl(Z\times_{X\sqcup Y}X\bigr)
\sqcup
\bigl(Z\times_{X\sqcup Y}Y\bigr).
\]

One may also view this as the localic shadow of the bicategorical extensivity of the $2$-category of Grothendieck toposes, established by Bunge and Lack~\cite{bunge2003van}. Under the correspondence sending a locale to its category of sheaves, the preceding decomposition is reflected by the corresponding bicategorical coproduct and pseudopullback decomposition for localic Grothendieck toposes. This gives a conceptual explanation of the extensivity of $\Loc$, while the frame-level argument above gives the elementary verification.

The Cantor-space construction above also shows that, despite being extensive, $\Loc$ is \emph{not} doubly-infinitary distributive. Indeed, the same comparison map is obtained from the discrete locale on the set $2^{\NN}$ to the localic Cantor space. Since both locales are spatial, an isomorphism between them in $\Loc$ would induce a homeomorphism between their spaces of points. This is impossible by the argument above. Hence $\Loc$ is extensive, but not doubly-infinitary distributive.

\subsection{Doubly-infinitary distributive categories that are not cartesian closed}

Theorem~\ref{thm:ccc} shows that \emph{free} doubly-infinitary distributive categories are cartesian closed. In general, however, doubly-infinitary distributivity does not imply cartesian closedness. We first record a useful criterion comparing cartesian closedness of $\catC$ with that of $\Fam(\catC)$ under ordinary infinitary distributivity.

\begin{theorem}\label{theo:cartesian-closed-infinitary-distributive-Fam}
	Let $\catC$ be an infinitary distributive category. Then $\catC$ is cartesian closed if and only if $\Fam(\catC)$ is cartesian closed.
\end{theorem}

\begin{proof}
	If $\catC$ is cartesian closed, then so is $\Fam(\catC)$; see, for instance,~\cite{zbMATH07186728, LV24b}.
	
	Conversely, assume that $\Fam(\catC)$ is cartesian closed. The inclusion $\eta^\Sigma_\catC:\catC\to\Fam(\catC)$ is fully faithful and has left adjoint
	\[
	\coprod:\Fam(\catC)\to\catC.
	\]
	Infinitary distributivity of $\catC$ says precisely that $\coprod$ preserves finite products. By~\cite[Proposition~4.3.1]{johnstone2002sketches}, it follows that $\catC$ is an exponential ideal in $\Fam(\catC)$, and hence is cartesian closed.
\end{proof}

\begin{corollary}\label{coro:example-creator}
	Let $\catC$ be an infinitary distributive category with small products and small coproducts, and suppose that $\catC$ is not cartesian closed. Then:
	\begin{enumerate}
		\item $\Fam(\catC)$ is doubly-infinitary distributive;
		\item $\Fam(\catC)$ is not cartesian closed.
	\end{enumerate}
\end{corollary}

\begin{proof}
	This follows directly from Lemma~\ref{ex:famprod} and Theorem~\ref{theo:cartesian-closed-infinitary-distributive-Fam}.
\end{proof}

\begin{example}
	{\footnotesize\emph{Corrected from the published form.}}\par\smallskip
	The category $\Top$ is not cartesian closed~\cite[Proposition~7.1.2]{borceux1994handbook}. Since $\Top$ is infinitary distributive, Corollary~\ref{coro:example-creator} implies that $\Fam(\Top)$ is doubly-infinitary distributive but not cartesian closed.
\end{example}

\begin{example}\label{example:loc}
	Let $\ConTop$ be the category of connected locally connected topological spaces and continuous maps, and let $\LocConTop$ be the category of locally connected spaces. Every locally connected space is a coproduct of its connected components, and these components are connected and locally connected; hence
	\[
	\LocConTop\simeq \Fam(\ConTop)
	\]
	by~\cite[Proposition~6.15]{zbMATH05306928}. Since $\ConTop$ has small products, computed by the usual product topology, $\LocConTop$ is doubly-infinitary distributive by Lemma~\ref{ex:famprod}.
	
	\smallskip
	\noindent\emph{Despite being doubly-infinitary distributive, $\LocConTop$ is not cartesian closed.}
\end{example}

The preceding example admits a conceptual generalisation using connected objects. Recall that, in a category $\catC$ with small coproducts, an object $C$ is called \emph{connected} if the representable functor
\[
\catC(C,-):\catC\to\Set
\]
preserves small coproducts. We denote by $\mathsf{Con}\catC$ the full subcategory of $\catC$ spanned by the connected objects.

\begin{lemmma}\label{lem:in-terms-of-connected}
	Let $\catC$ be an extensive category with small coproducts. Suppose that $\mathsf{Con}\catC$ has small products and that every object of $\catC$ is isomorphic to a coproduct of connected objects. Then $\catC$ is doubly-infinitary distributive.
\end{lemmma}

\begin{proof}
	By extensivity and~\cite[Proposition~6.15]{zbMATH05306928}, one has
	\[
	\catC\simeq \Fam(\mathsf{Con}\catC).
	\]
	Since $\mathsf{Con}\catC$ has small products, Lemma~\ref{ex:famprod} applies.
\end{proof}

\subsection{Cartesian closed categories need not be doubly-infinitary distributive}

Any cartesian closed category with small coproducts is infinitary distributive, since $A\times-$ is a left adjoint. This does not extend to doubly-infinitary distributivity.

\begin{counterexample}
	The category $\FinSet$ of finite sets is cartesian closed, but not doubly-infinitary distributive, since it lacks infinite products and infinite coproducts.
\end{counterexample}

\begin{counterexample}[Pseudotopological spaces]
	Let $\PsTop$ be the category of pseudotopological spaces; see, for instance,~\cite{hofmann2014monoidal}. It is a complete and cocomplete quasitopos, hence cartesian closed.
	
	The inclusion $\Top\hookrightarrow\PsTop$ is fully faithful and creates the products and coproducts of diagrams of topological spaces. If $\PsTop$ were doubly-infinitary distributive, Lemma~\ref{lem:conservative-transfer} would imply that $\Top$ is doubly-infinitary distributive, contradicting the counterexample above. Therefore $\PsTop$ is cartesian closed, complete and cocomplete, but not doubly-infinitary distributive.
\end{counterexample}

\begin{counterexample}
	{\footnotesize\emph{Corrected from the published form.}}\par\smallskip
	The category $\Qbs$ of quasi-Borel spaces~\cite{heunen2017convenient} provides another example.
	
	The category $\Qbs$ is a complete and cocomplete quasitopos and hence cartesian closed. It is of interest as a setting for probability theory. The canonical functor from standard Borel spaces to $\Qbs$ is fully faithful and preserves countable products and countable coproducts~\cite[Propositions~15 and~19]{heunen2017convenient}.
	
	Nevertheless, $\Qbs$ is not doubly-infinitary distributive; see the computation in the following remark.
\end{counterexample}

\begin{remark}[Details for quasi-Borel spaces]
	A quasi-Borel space $X$ consists of a set $|X|$ together with a set $M_X\subseteq \Set(\RR,|X|)$ satisfying:
	\begin{enumerate}
		\item $M_X$ contains all constant functions;
		\item $M_X$ is closed under precomposition with measurable functions $\RR\to\RR$;
		\item $M_X$ is closed under gluing along countable measurable partitions of $\RR$.
	\end{enumerate}
	A morphism $f:X\to Y$ is a function $|X|\to |Y|$ such that $f\circ g\in M_Y$ whenever $g\in M_X$.
	
	Products and coproducts are given by
	\[
	\left|\prod_{j\in J}X_j\right|=\prod_{j\in J}|X_j|,
	\qquad
	M_{\prod_{j\in J}X_j}
	=
	\Bigl\{
	\prodfam{f_j}{j\in J}
	\ \Bigm|\
	f_j\in M_{X_j}
	\Bigr\},
	\]
	and
	\[
	\left|\coprod_{j\in J}X_j\right|=\coprod_{j\in J}|X_j|,
	\]
	\[
	M_{\coprod_{j\in J}X_j}
	=
	\Bigl\{
	\begin{aligned}
	&\, \lambda r.\,\iota_{j(n(r))}(f_{n(r)}(r))\\
	&\Bigm|\;
	n : \RR \to \NN \text{ measurable},\;
	j : \NN \to J,\;
	\forall m \in \NN.\ f_m \in M_{X_{j(m)}}
	\Bigr\}.
	\end{aligned}
	\]
	In particular, elements of $M_{\coprod_{j\in J}X_j}$ always factor through countably many coproduct components. When $J$ is countable and is equipped with its discrete $\sigma$-algebra, this simplifies to
	\[
	M_{\coprod_{j \in J} X_j}
	=
	\Bigl\{
	\, \lambda r.\,\iota_{k(r)}(f_{k(r)}(r))
	\ \Bigm|\
	k : \RR \to J \text{ measurable},\;
	(f_j\in M_{X_j})_{j\in\operatorname{im}(k)}
	\Bigr\}.
	\]
	
	Consider the canonical comparison
	\[
	d :
	\left(\coprod_{f : \prod_{j \in J} I_j} \prod_{j \in J} C_{j,f(j)}\right)
	\longrightarrow
	\left(\prod_{j \in J} \coprod_{i \in I_j} C_{ji}\right)
	\]
	in $\Qbs$. Its underlying function is a bijection, as in $\Set$. The obstruction is whether the inverse $d^{-1}$ is a morphism in $\Qbs$, i.e.\ whether $d^{-1}\circ g$ is always an element of
	\[
	M_{\coprod_{f : \prod_{j \in J} I_j} \prod_{j \in J} C_{j,f(j)}}
	\]
	for every
	\[
	g\in M_{\prod_{j \in J} \coprod_{i \in I_j} C_{ji}}.
	\]
	Take $J=\NN$, $I_j=2$, and $C_{j,i}=1$ for all $j\in\NN$ and $i\in 2$. Let $(q_j)_{j\in\NN}$ enumerate $\mathbb Q$ and define
	\[
	g(r):=
	\bigl(\mathbf{1}_{\{q_j<r\}}\bigr)_{j\in\NN}.
	\]
	Each coordinate of $g$ is measurable, so $g$ is a random element of the product. Moreover, $g$ is injective: if $r<s$, one may choose $q_j$ with $r<q_j<s$, in which case the $j$th coordinates of $g(r)$ and $g(s)$ differ. Hence $g$ has uncountable image.
	
	Under the evident identification
	\[
	\prod_{j\in\NN}(1+1)\cong 2^{\NN},
	\]
	the coproduct component of $(d^{-1}\circ g)(r)$ is precisely $g(r)$. Thus $d^{-1}\circ g$ meets uncountably many coproduct components. Since every random element of a coproduct in $\Qbs$ factors through countably many components, $d^{-1}\circ g$ is not random. Therefore $d^{-1}$ is not a morphism in $\Qbs$. Hence $\Qbs$ is infinitary distributive but not doubly-infinitary distributive.
\end{remark}

\subsection{Categories of categorical structures}\label{sec:categorical-structures}

Consider $\mathbf{Cat}$, the category of small categories. One has
\[
\mathbf{Cat}\simeq \Fam(\ConCat),
\]
where $\ConCat$ is the full subcategory of connected categories.

\begin{remark}\label{rem:ConCat-not-product-complete}
	Although $\ConCat$ has finite products, it is not closed under small products. For example, let $C_n$ be the thin category on objects $\{0,1,\dots,n\}$ with order relations forming the fence
	\[
	0<1>2<3>\cdots.
	\]
	Each $C_n$ is connected. In the product $\prod_{n\in\NN}C_n$, the objects
	\[
	(0,0,0,\dots)
	\qquad\text{and}\qquad
	(0,1,2,3,\dots)
	\]
	cannot be joined by any finite zigzag: any such zigzag would have length at least $n$ in the $n$th coordinate. Hence the product is not connected.
	
	By contrast, the category of connected groupoids is closed under small products. Therefore the category of groupoids is doubly-infinitary distributive by Lemma~\ref{lem:in-terms-of-connected}.
\end{remark}

Nevertheless, $\mathbf{Cat}$ itself is doubly-infinitary distributive.

\begin{example}[$\mathbf{Cat}$, $\Pos$, $\wCpo$]
	The forgetful functor $\mathbf{Cat}\to \mathbf{Graph}$ creates small products and small coproducts. Since $\mathbf{Graph}$ is doubly-infinitary distributive by Example~\ref{ex:presheaves}, Lemma~\ref{lem:conservative-transfer} implies that $\mathbf{Cat}$ is doubly-infinitary distributive.
	
	The same argument applies to the category $\Pos$ of posets. Likewise, the forgetful functor $\wCpo\to\Pos$ creates small products and small coproducts, so $\wCpo$ is doubly-infinitary distributive as well.
\end{example}

\section{Final remarks and future work}\label{sec:final-remarks}

We have introduced doubly-infinitary distributive categories and studied the free doubly-infinitary distributive category
\[
\Dist(\catC)=\Fam\!\bigl(\Fam(\catC^{op})^{op}\bigr)
\]
on a category $\catC$, showing that it is always cartesian closed (Theorem~\ref{thm:ccc}).

Replacing $\Fam$ by its finitary variant $\FinFam$ yields the familiar distributive constructions. The distributive law
\[
\FinFam(\FinFam(\catC)^{op})^{op}\ \to\ \FinFam(\FinFam(\catC^{op})^{op})
\]
gives rise to finitary distributive categories, namely categories with finite products and finite coproducts satisfying the usual distributivity condition. Similarly, the distributive law
\[
\FinFam(\Fam(\catC)^{op})^{op}\ \to\ \Fam(\FinFam(\catC^{op})^{op})
\]
gives rise to infinitary distributive categories in the classical sense, namely categories with finite products and small coproducts in which finite products distribute over small coproducts.

As a small variation on our results, one obtains that
\[
\FinDist{\catC}=\FinFam\bigl((\FinFam(\catC^{op}))^{op}\bigr)
\]
is cartesian closed for any \emph{locally finite} category $\catC$.
\subsection{Further properties of \texorpdfstring{$\Dist(\catC)$}{Dist(C)}}

The category $\Dist(\catC)$ enjoys further properties inherited from being a free coproduct completion $\Fam(\catD)$ with $\catD=\Fam(\catC^{op})^{op}$. In particular, $\Dist(\catC)$ is extensive, and more generally satisfies standard exactness properties associated with $\Fam$; see~\cite{zbMATH07186728, zbMATH05306928, carboni1993introduction, nunes2024free, nunes2023chad, nunes2025unravelingiterativechad, zbMATH07799814, LV24b}.

\subsection{Non-canonical isomorphisms}

Definition~\ref{def:doubly-infinitary-distributivity} uses the invertibility of a canonical comparison. It is natural to ask whether the existence of a \emph{non-canonical} natural isomorphism suffices; compare Pisani's question for distributive categories, for instance~\cite{zbMATH06039246}.

The answer is affirmative. Using the general framework of~\cite{zbMATH07041646}, one sees that a category $\catC$ is doubly-infinitary distributive if it has products and coproducts and, for each family $(C_{ji})_{j\in J,\ i\in I_j}$, there exists a natural isomorphism
\[
\left(\coprod_{f\in \prod_{j\in J} I_j} \prod_{j\in J}C_{j,f(j)}\right)
\ \to\
\left(\prod_{j\in J} \coprod_{i\in I_j} C_{ji}\right).
\]
Equivalently, since a product- and coproduct-complete category is doubly-infinitary distributive if and only if $\coprod:\Fam(\catC)\to\catC$ preserves products (Remark~\ref{rem:concise-definition}), this follows from general results on preservation of limits from naturality of isomorphisms; see, for instance,~\cite{zbMATH06209910, zbMATH07041646}.

\subsection{Generalised categorical structures}

It was recently shown that categories of $(T,\catV)$-categories are extensive under broad hypotheses; see~\cite{zbMATH07377652}. The examples in Section~\ref{sec:examples} show that doubly-infinitary distributivity is subtler: it holds for $\mathbf{Cat}$ and for locally connected topological spaces, but fails for $\Top$.

This motivates the following problem.

\begin{openquestion}
	Under which conditions on $T$ and $\catV$ does the category of $(T,\catV)$-categories exhibit doubly-infinitary distributivity?
\end{openquestion}

\subsection{G\"odel's Dialectica interpretation}

The exponential formula of Theorem~\ref{thm:ccc} is a variation on the categorical manifestations of G\"odel's Dialectica interpretation~\cite{godel1958bisher, Hyland02, zbMATH04104952, zbMATH07648691}, on exponentials in categories of polynomials and containers~\cite{altenkirch2010higher}, and on Diller--Nahm style formulas~\cite{diller1974variante, Hyland02, nunes2023chad}. The precise relationship is discussed in the separate paper~\cite{LV24b}.

\section*{Acknowledgements}

The authors thank the anonymous referee for a careful reading of the paper and for several useful remarks that helped us improve the presentation. They are also grateful to Mat\'ias Menni for a helpful question concerning the extensivity of the category of locales, which led us to include a more precise discussion of this point.

This research was supported by the NWO Veni grant VI.Veni.202.124 and by the ERC project FoRECAST. It was also supported through the programme ``Oberwolfach Leibniz Fellows'' of the Mathematisches Forschungsinstitut Oberwolfach in 2022.

The first author gratefully acknowledges support from the Fields Institute for Research in Mathematical Sciences through a Fields Research Fellowship in 2023. He was also supported by the Centre for Mathematics of the University of Coimbra (CMUC, \url{https://doi.org/10.54499/UID/00324/2025}) under the Funda\c{c}\~ao para a Ci\^encia e a Tecnologia (FCT), through the grants UID/00324/2025 and UID/PRR/00324/2025. He also expresses his sincere gratitude to Henrique Bursztyn and to the Instituto Nacional de Matem\'atica Pura e Aplicada (IMPA) for their generous hospitality.

\section*{Statements and Declarations}

\subsection*{Ethics declaration}

Not applicable.

\subsection*{Funding}

This research was supported by the NWO Veni grant VI.Veni.202.124, the ERC project FoRECAST, and the programme ``Oberwolfach Leibniz Fellows'' of the Mathematisches Forschungsinstitut Oberwolfach. The first author was supported by the Fields Institute for Research in Mathematical Sciences through a Fields Research Fellowship in 2023, and by the Centre for Mathematics of the University of Coimbra (CMUC) under the Funda\c{c}\~ao para a Ci\^encia e a Tecnologia (FCT), through the grants UID/00324/2025 and UID/PRR/00324/2025.

\subsection*{Competing interests}

The authors declare that they have no competing interests.

\subsection*{Data availability}

Data sharing is not applicable to this article as no datasets were generated or analysed during the current study.

% The resolved bibliography is included directly in the stand-alone arXiv source.
\providecommand{\bibcommenthead}{}
\newenvironment{barticle}{}{}
\newenvironment{bbook}{}{}
\newenvironment{bchapter}{}{}
\newenvironment{botherref}{}{}


\begin{thebibliography}{47}
\ifx \bisbn   \undefined \def \bisbn  #1{ISBN #1}\fi
\ifx \binits  \undefined \def \binits#1{#1}\fi
\ifx \bauthor  \undefined \def \bauthor#1{#1}\fi
\ifx \batitle  \undefined \def \batitle#1{#1}\fi
\ifx \bjtitle  \undefined \def \bjtitle#1{#1}\fi
\ifx \bvolume  \undefined \def \bvolume#1{\textbf{#1}}\fi
\ifx \byear  \undefined \def \byear#1{#1}\fi
\ifx \bissue  \undefined \def \bissue#1{#1}\fi
\ifx \bfpage  \undefined \def \bfpage#1{#1}\fi
\ifx \blpage  \undefined \def \blpage #1{#1}\fi
\ifx \burl  \undefined \def \burl#1{\textsf{#1}}\fi
\ifx \doiurl  \undefined \def \doiurl#1{\url{https://doi.org/#1}}\fi
\ifx \betal  \undefined \def \betal{\textit{et al.}}\fi
\ifx \binstitute  \undefined \def \binstitute#1{#1}\fi
\ifx \binstitutionaled  \undefined \def \binstitutionaled#1{#1}\fi
\ifx \bctitle  \undefined \def \bctitle#1{#1}\fi
\ifx \beditor  \undefined \def \beditor#1{#1}\fi
\ifx \bpublisher  \undefined \def \bpublisher#1{#1}\fi
\ifx \bbtitle  \undefined \def \bbtitle#1{#1}\fi
\ifx \bedition  \undefined \def \bedition#1{#1}\fi
\ifx \bseriesno  \undefined \def \bseriesno#1{#1}\fi
\ifx \blocation  \undefined \def \blocation#1{#1}\fi
\ifx \bsertitle  \undefined \def \bsertitle#1{#1}\fi
\ifx \bsnm \undefined \def \bsnm#1{#1}\fi
\ifx \bsuffix \undefined \def \bsuffix#1{#1}\fi
\ifx \bparticle \undefined \def \bparticle#1{#1}\fi
\ifx \barticle \undefined \def \barticle#1{#1}\fi
\bibcommenthead
\ifx \bconfdate \undefined \def \bconfdate #1{#1}\fi
\ifx \botherref \undefined \def \botherref #1{#1}\fi
\ifx \url \undefined \def \url#1{\textsf{#1}}\fi
\ifx \bchapter \undefined \def \bchapter#1{#1}\fi
\ifx \bbook \undefined \def \bbook#1{#1}\fi
\ifx \bcomment \undefined \def \bcomment#1{#1}\fi
\ifx \oauthor \undefined \def \oauthor#1{#1}\fi
\ifx \citeauthoryear \undefined \def \citeauthoryear#1{#1}\fi
\ifx \endbibitem  \undefined \def \endbibitem {}\fi
\ifx \bconflocation  \undefined \def \bconflocation#1{#1}\fi
\ifx \arxivurl  \undefined \def \arxivurl#1{\textsf{#1}}\fi
\csname PreBibitemsHook\endcsname

%%% 1
\bibitem[\protect\citeauthoryear{Carboni
  et~al.}{1993}]{carboni1993introduction}
\begin{barticle}
\bauthor{\bsnm{Carboni}, \binits{A.}},
\bauthor{\bsnm{Lack}, \binits{S.}},
\bauthor{\bsnm{Walters}, \binits{R.F.C.}}:
\batitle{Introduction to extensive and distributive categories}.
\bjtitle{Journal of Pure and Applied Algebra}
\bvolume{84}(\bissue{2}),
\bfpage{145}--\blpage{158}
(\byear{1993})
\doiurl{10.1016/0022-4049(93)90035-R}
\end{barticle}
\endbibitem

%%% 2
\bibitem[\protect\citeauthoryear{Marmolejo
  et~al.}{2012}]{marmolejo2012completely}
\begin{barticle}
\bauthor{\bsnm{Marmolejo}, \binits{F.}},
\bauthor{\bsnm{Rosebrugh}, \binits{R.}},
\bauthor{\bsnm{Wood}, \binits{R.J.}}:
\batitle{Completely and totally distributive categories {I}}.
\bjtitle{Journal of Pure and Applied Algebra}
\bvolume{216}(\bissue{8--9}),
\bfpage{1775}--\blpage{1790}
(\byear{2012})
\end{barticle}
\endbibitem

%%% 3
\bibitem[\protect\citeauthoryear{Borceux and Janelidze}{2008}]{zbMATH05306928}
\begin{bbook}
\bauthor{\bsnm{Borceux}, \binits{F.}},
\bauthor{\bsnm{Janelidze}, \binits{G.}}:
\bbtitle{Galois Theories},
\bedition{Paperback reprint} edn.
\bsertitle{Cambridge Studies in Advanced Mathematics},
vol. \bseriesno{72},
p. \bfpage{341}.
\bpublisher{Cambridge University Press},
\blocation{Cambridge}
(\byear{2008}).
\bcomment{Paperback reprint of the 2001 hardback edition}
\end{bbook}
\endbibitem

%%% 4
\bibitem[\protect\citeauthoryear{Ad{\'a}mek and
  Rosick{\'y}}{2020}]{zbMATH07186728}
\begin{barticle}
\bauthor{\bsnm{Ad{\'a}mek}, \binits{J.}},
\bauthor{\bsnm{Rosick{\'y}}, \binits{J.}}:
\batitle{How nice are free completions of categories?}
\bjtitle{Topology and its Applications}
\bvolume{273},
\bfpage{106972}
(\byear{2020})
\doiurl{10.1016/j.topol.2019.106972}
\end{barticle}
\endbibitem

%%% 5
\bibitem[\protect\citeauthoryear{Lucatelli~Nunes and
  V{\'a}k{\'a}r}{2023}]{nunes2023chad}
\begin{barticle}
\bauthor{\bsnm{Lucatelli~Nunes}, \binits{F.}},
\bauthor{\bsnm{V{\'a}k{\'a}r}, \binits{M.}}:
\batitle{{CHAD} for expressive total languages}.
\bjtitle{Mathematical Structures in Computer Science}
\bvolume{33}(\bissue{4--5}),
\bfpage{311}--\blpage{426}
(\byear{2023})
\end{barticle}
\endbibitem

%%% 6
\bibitem[\protect\citeauthoryear{Prezado}{2024}]{zbMATH07799814}
\begin{barticle}
\bauthor{\bsnm{Prezado}, \binits{R.}}:
\batitle{On effective descent {$\mathcal{V}$}-functors and familial descent
  morphisms}.
\bjtitle{Journal of Pure and Applied Algebra}
\bvolume{228}(\bissue{5}),
\bfpage{107597}
(\byear{2024})
\doiurl{10.1016/j.jpaa.2023.107597}
\end{barticle}
\endbibitem

%%% 7
\bibitem[\protect\citeauthoryear{Blackwell et~al.}{1989}]{zbMATH04105188}
\begin{barticle}
\bauthor{\bsnm{Blackwell}, \binits{R.}},
\bauthor{\bsnm{Kelly}, \binits{G.M.}},
\bauthor{\bsnm{Power}, \binits{A.J.}}:
\batitle{Two-dimensional monad theory}.
\bjtitle{Journal of Pure and Applied Algebra}
\bvolume{59}(\bissue{1}),
\bfpage{1}--\blpage{41}
(\byear{1989})
\doiurl{10.1016/0022-4049(89)90160-6}
\end{barticle}
\endbibitem

%%% 8
\bibitem[\protect\citeauthoryear{Lucatelli~Nunes}{2016}]{MR3491845}
\begin{barticle}
\bauthor{\bsnm{Lucatelli~Nunes}, \binits{F.}}:
\batitle{On biadjoint triangles}.
\bjtitle{Theory and Applications of Categories}
\bvolume{31}(\bissue{9}),
\bfpage{217}--\blpage{256}
(\byear{2016})
\end{barticle}
\endbibitem

%%% 9
\bibitem[\protect\citeauthoryear{Lucatelli~Nunes}{2018}]{zbMATH06970806}
\begin{barticle}
\bauthor{\bsnm{Lucatelli~Nunes}, \binits{F.}}:
\batitle{On lifting of biadjoints and lax algebras}.
\bjtitle{Categories and General Algebraic Structures with Applications}
\bvolume{9}(\bissue{1}),
\bfpage{29}--\blpage{58}
(\byear{2018})
\end{barticle}
\endbibitem

%%% 10
\bibitem[\protect\citeauthoryear{Marmolejo}{2004}]{zbMATH02116176}
\begin{barticle}
\bauthor{\bsnm{Marmolejo}, \binits{F.}}:
\batitle{Distributive laws for pseudomonads. {II}}.
\bjtitle{Journal of Pure and Applied Algebra}
\bvolume{194}(\bissue{1--2}),
\bfpage{169}--\blpage{182}
(\byear{2004})
\doiurl{10.1016/j.jpaa.2004.04.008}
\end{barticle}
\endbibitem

%%% 11
\bibitem[\protect\citeauthoryear{Marmolejo and Wood}{2008}]{zbMATH05256222}
\begin{barticle}
\bauthor{\bsnm{Marmolejo}, \binits{F.}},
\bauthor{\bsnm{Wood}, \binits{R.J.}}:
\batitle{Coherence for pseudodistributive laws revisited}.
\bjtitle{Theory and Applications of Categories}
\bvolume{20},
\bfpage{74}--\blpage{84}
(\byear{2008})
\end{barticle}
\endbibitem

%%% 12
\bibitem[\protect\citeauthoryear{Marmolejo}{1999}]{marmolejo1999distributive}
\begin{barticle}
\bauthor{\bsnm{Marmolejo}, \binits{F.}}:
\batitle{Distributive laws for pseudomonads}.
\bjtitle{Theory and Applications of Categories}
\bvolume{5}(\bissue{5}),
\bfpage{91}--\blpage{147}
(\byear{1999})
\end{barticle}
\endbibitem

%%% 13
\bibitem[\protect\citeauthoryear{von Glehn}{2018}]{von2018polynomials}
\begin{barticle}
\bauthor{\bsnm{Glehn}, \binits{T.}}:
\batitle{Polynomials, fibrations and distributive laws}.
\bjtitle{Theory and Applications of Categories}
\bvolume{33}(\bissue{36}),
\bfpage{1111}--\blpage{1144}
(\byear{2018})
\end{barticle}
\endbibitem

%%% 14
\bibitem[\protect\citeauthoryear{Moss and von Glehn}{2018}]{glehnmoss2018}
\begin{bchapter}
\bauthor{\bsnm{Moss}, \binits{S.K.}},
\bauthor{\bsnm{Glehn}, \binits{T.}}:
\bctitle{Dialectica models of type theory}.
In: \beditor{\bsnm{Dawar}, \binits{A.}},
\beditor{\bsnm{Gr{\"a}del}, \binits{E.}} (eds.)
\bbtitle{Proceedings of the 33rd Annual ACM/IEEE Symposium on Logic in Computer
  Science}.
\bsertitle{LICS '18},
pp. \bfpage{739}--\blpage{748}.
\bpublisher{ACM},
\blocation{New York, NY, USA}
(\byear{2018}).
\doiurl{10.1145/3209108.3209207}
\end{bchapter}
\endbibitem

%%% 15
\bibitem[\protect\citeauthoryear{V{\'a}k{\'a}r}{2017}]{vakar2017search}
\begin{botherref}
\oauthor{\bsnm{V{\'a}k{\'a}r}, \binits{M.}}:
In search of effectful dependent types.
PhD thesis,
University of Oxford
(2017).
DPhil thesis.
\url{https://arxiv.org/abs/1706.07997}
\end{botherref}
\endbibitem

%%% 16
\bibitem[\protect\citeauthoryear{V{\'a}k{\'a}r}{2015}]{vakar2015categorical}
\begin{bchapter}
\bauthor{\bsnm{V{\'a}k{\'a}r}, \binits{M.}}:
\bctitle{A categorical semantics for linear logical frameworks}.
In: \bbtitle{Foundations of Software Science and Computation Structures}.
\bsertitle{Lecture Notes in Computer Science},
vol. \bseriesno{9034},
pp. \bfpage{102}--\blpage{116}.
\bpublisher{Springer},
\blocation{Berlin, Heidelberg}
(\byear{2015})
\end{bchapter}
\endbibitem

%%% 17
\bibitem[\protect\citeauthoryear{Kock}{1995}]{kock1995monads}
\begin{barticle}
\bauthor{\bsnm{Kock}, \binits{A.}}:
\batitle{Monads for which structures are adjoint to units}.
\bjtitle{Journal of Pure and Applied Algebra}
\bvolume{104}(\bissue{1}),
\bfpage{41}--\blpage{59}
(\byear{1995})
\end{barticle}
\endbibitem

%%% 18
\bibitem[\protect\citeauthoryear{Marmolejo}{1997}]{zbMATH01024330}
\begin{barticle}
\bauthor{\bsnm{Marmolejo}, \binits{F.}}:
\batitle{Doctrines whose structure forms a fully faithful adjoint string}.
\bjtitle{Theory and Applications of Categories}
\bvolume{3},
\bfpage{24}--\blpage{44}
(\byear{1997})
\end{barticle}
\endbibitem

%%% 19
\bibitem[\protect\citeauthoryear{Street}{1980}]{zbMATH03680046}
\begin{barticle}
\bauthor{\bsnm{Street}, \binits{R.}}:
\batitle{Fibrations in bicategories}.
\bjtitle{Cahiers de Topologie et G{\'e}om{\'e}trie Diff{\'e}rentielle
  Cat{\'e}goriques}
\bvolume{21},
\bfpage{111}--\blpage{159}
(\byear{1980})
\end{barticle}
\endbibitem

%%% 20
\bibitem[\protect\citeauthoryear{Clementino and
  Lucatelli~Nunes}{2024}]{clementino2023lax}
\begin{barticle}
\bauthor{\bsnm{Clementino}, \binits{M.M.}},
\bauthor{\bsnm{Lucatelli~Nunes}, \binits{F.}}:
\batitle{Lax comma 2-categories and admissible 2-functors}.
\bjtitle{Theory and Applications of Categories}
\bvolume{40},
\bfpage{180}--\blpage{226}
(\byear{2024})
\end{barticle}
\endbibitem

%%% 21
\bibitem[\protect\citeauthoryear{Gray}{1966}]{MR0213413}
\begin{bchapter}
\bauthor{\bsnm{Gray}, \binits{J.W.}}:
\bctitle{Fibred and cofibred categories}.
In: \bbtitle{Proceedings of the Conference on Categorical Algebra, La Jolla,
  1965},
pp. \bfpage{21}--\blpage{83}.
\bpublisher{Springer},
\blocation{New York}
(\byear{1966})
\end{bchapter}
\endbibitem

%%% 22
\bibitem[\protect\citeauthoryear{Walker}{2019}]{walker2019distributive}
\begin{barticle}
\bauthor{\bsnm{Walker}, \binits{C.}}:
\batitle{Distributive laws via admissibility}.
\bjtitle{Applied Categorical Structures}
\bvolume{27}(\bissue{6}),
\bfpage{567}--\blpage{617}
(\byear{2019})
\end{barticle}
\endbibitem

%%% 23
\bibitem[\protect\citeauthoryear{Street}{1987}]{zbMATH04008629}
\begin{barticle}
\bauthor{\bsnm{Street}, \binits{R.}}:
\batitle{Correction to ``{Fibrations} in bicategories''}.
\bjtitle{Cahiers de Topologie et G{\'e}om{\'e}trie Diff{\'e}rentielle
  Cat{\'e}goriques}
\bvolume{28}(\bissue{1}),
\bfpage{53}--\blpage{56}
(\byear{1987})
\end{barticle}
\endbibitem

%%% 24
\bibitem[\protect\citeauthoryear{Lucatelli~Nunes}{2018}]{zbMATH06881682}
\begin{barticle}
\bauthor{\bsnm{Lucatelli~Nunes}, \binits{F.}}:
\batitle{Pseudo-{Kan} extensions and descent theory}.
\bjtitle{Theory and Applications of Categories}
\bvolume{33},
\bfpage{390}--\blpage{444}
(\byear{2018})
\end{barticle}
\endbibitem

%%% 25
\bibitem[\protect\citeauthoryear{Abbott et~al.}{2003}]{abbott2003categories}
\begin{bchapter}
\bauthor{\bsnm{Abbott}, \binits{M.}},
\bauthor{\bsnm{Altenkirch}, \binits{T.}},
\bauthor{\bsnm{Ghani}, \binits{N.}}:
\bctitle{Categories of containers}.
In: \bbtitle{Foundations of Software Science and Computation Structures}.
\bsertitle{Lecture Notes in Computer Science},
vol. \bseriesno{2620},
pp. \bfpage{23}--\blpage{38}.
\bpublisher{Springer},
\blocation{Berlin, Heidelberg}
(\byear{2003})
\end{bchapter}
\endbibitem

%%% 26
\bibitem[\protect\citeauthoryear{Altenkirch
  et~al.}{2010}]{altenkirch2010higher}
\begin{bchapter}
\bauthor{\bsnm{Altenkirch}, \binits{T.}},
\bauthor{\bsnm{Levy}, \binits{P.}},
\bauthor{\bsnm{Staton}, \binits{S.}}:
\bctitle{Higher-order containers}.
In: \bbtitle{Programs, Proofs, Processes}.
\bsertitle{Lecture Notes in Computer Science},
vol. \bseriesno{6158},
pp. \bfpage{11}--\blpage{20}.
\bpublisher{Springer},
\blocation{Berlin, Heidelberg}
(\byear{2010})
\end{bchapter}
\endbibitem

%%% 27
\bibitem[\protect\citeauthoryear{Fawcett and Wood}{1990}]{zbMATH04136008}
\begin{barticle}
\bauthor{\bsnm{Fawcett}, \binits{B.}},
\bauthor{\bsnm{Wood}, \binits{R.J.}}:
\batitle{Constructive complete distributivity. {I}}.
\bjtitle{Mathematical Proceedings of the Cambridge Philosophical Society}
\bvolume{107}(\bissue{1}),
\bfpage{81}--\blpage{89}
(\byear{1990})
\doiurl{10.1017/S0305004100068377}
\end{barticle}
\endbibitem

%%% 28
\bibitem[\protect\citeauthoryear{Clementino et~al.}{1996}]{zbMATH01002289}
\begin{barticle}
\bauthor{\bsnm{Clementino}, \binits{M.M.}},
\bauthor{\bsnm{Giuli}, \binits{E.}},
\bauthor{\bsnm{Tholen}, \binits{W.}}:
\batitle{Topology in a category: {Compactness}}.
\bjtitle{Portugaliae Mathematica}
\bvolume{53}(\bissue{4}),
\bfpage{397}--\blpage{433}
(\byear{1996})
\end{barticle}
\endbibitem

%%% 29
\bibitem[\protect\citeauthoryear{Picado and Pultr}{2012}]{zbMATH05898723}
\begin{bbook}
\bauthor{\bsnm{Picado}, \binits{J.}},
\bauthor{\bsnm{Pultr}, \binits{A.}}:
\bbtitle{Frames and Locales: Topology Without Points}.
\bsertitle{Frontiers in Mathematics}.
\bpublisher{Birkh{\"a}user},
\blocation{Basel}
(\byear{2012}).
\doiurl{10.1007/978-3-0348-0154-6}
\end{bbook}
\endbibitem

%%% 30
\bibitem[\protect\citeauthoryear{Taylor}{1999}]{taylor1999practical}
\begin{bbook}
\bauthor{\bsnm{Taylor}, \binits{P.}}:
\bbtitle{Practical Foundations of Mathematics}.
\bsertitle{Cambridge Studies in Advanced Mathematics},
vol. \bseriesno{59}.
\bpublisher{Cambridge University Press},
\blocation{Cambridge}
(\byear{1999})
\end{bbook}
\endbibitem

%%% 31
\bibitem[\protect\citeauthoryear{Bunge and Lack}{2003}]{bunge2003van}
\begin{barticle}
\bauthor{\bsnm{Bunge}, \binits{M.}},
\bauthor{\bsnm{Lack}, \binits{S.}}:
\batitle{Van {K}ampen theorems for toposes}.
\bjtitle{Advances in Mathematics}
\bvolume{179}(\bissue{2}),
\bfpage{291}--\blpage{317}
(\byear{2003})
\end{barticle}
\endbibitem

%%% 32
\bibitem[\protect\citeauthoryear{Lucatelli~Nunes and
  V{\'a}k{\'a}r}{2025}]{LV24b}
\begin{barticle}
\bauthor{\bsnm{Lucatelli~Nunes}, \binits{F.}},
\bauthor{\bsnm{V{\'a}k{\'a}r}, \binits{M.}}:
\batitle{Monoidal closure of {Grothendieck} constructions via
  {$\Sigma$}-tractable monoidal structures and {Dialectica} formulas}.
\bjtitle{Theory and Applications of Categories}
\bvolume{44},
\bfpage{1153}--\blpage{1217}
(\byear{2025})
\end{barticle}
\endbibitem

%%% 33
\bibitem[\protect\citeauthoryear{Johnstone}{2002}]{johnstone2002sketches}
\begin{bbook}
\bauthor{\bsnm{Johnstone}, \binits{P.T.}}:
\bbtitle{Sketches of an Elephant: A Topos Theory Compendium. Volume 2}.
\bsertitle{Oxford Logic Guides},
vol. \bseriesno{44}.
\bpublisher{Oxford University Press},
\blocation{Oxford}
(\byear{2002})
\end{bbook}
\endbibitem

%%% 34
\bibitem[\protect\citeauthoryear{Borceux}{1994}]{borceux1994handbook}
\begin{bbook}
\bauthor{\bsnm{Borceux}, \binits{F.}}:
\bbtitle{Handbook of Categorical Algebra 2: Categories and Structures}.
\bsertitle{Encyclopedia of Mathematics and its Applications},
vol. \bseriesno{51}.
\bpublisher{Cambridge University Press},
\blocation{Cambridge}
(\byear{1994})
\end{bbook}
\endbibitem

%%% 35
\bibitem[\protect\citeauthoryear{Hofmann et~al.}{2014}]{hofmann2014monoidal}
\begin{bbook}
\bauthor{\bsnm{Hofmann}, \binits{D.}},
\bauthor{\bsnm{Seal}, \binits{G.J.}},
\bauthor{\bsnm{Tholen}, \binits{W.}}:
\bbtitle{Monoidal Topology: A Categorical Approach to Order, Metric, and
  Topology}.
\bsertitle{Encyclopedia of Mathematics and its Applications},
vol. \bseriesno{153}.
\bpublisher{Cambridge University Press},
\blocation{Cambridge}
(\byear{2014})
\end{bbook}
\endbibitem

%%% 36
\bibitem[\protect\citeauthoryear{Heunen et~al.}{2017}]{heunen2017convenient}
\begin{bchapter}
\bauthor{\bsnm{Heunen}, \binits{C.}},
\bauthor{\bsnm{Kammar}, \binits{O.}},
\bauthor{\bsnm{Staton}, \binits{S.}},
\bauthor{\bsnm{Yang}, \binits{H.}}:
\bctitle{A convenient category for higher-order probability theory}.
In: \bbtitle{Proceedings of the 32nd Annual ACM/IEEE Symposium on Logic in
  Computer Science}.
\bsertitle{LICS '17},
pp. \bfpage{1}--\blpage{12}.
\bpublisher{IEEE},
\blocation{Piscataway, NJ}
(\byear{2017})
\end{bchapter}
\endbibitem

%%% 37
\bibitem[\protect\citeauthoryear{Lucatelli~Nunes et~al.}{2025a}]{nunes2024free}
\begin{barticle}
\bauthor{\bsnm{Lucatelli~Nunes}, \binits{F.}},
\bauthor{\bsnm{Prezado}, \binits{R.}},
\bauthor{\bsnm{V{\'a}k{\'a}r}, \binits{M.}}:
\batitle{Free extensivity via distributivity}.
\bjtitle{Portugaliae Mathematica}
\bvolume{82}(\bissue{1--2}),
\bfpage{177}--\blpage{204}
(\byear{2025})
\doiurl{10.4171/PM/2129}
\end{barticle}
\endbibitem

%%% 38
\bibitem[\protect\citeauthoryear{Lucatelli~Nunes
  et~al.}{2025b}]{nunes2025unravelingiterativechad}
\begin{botherref}
\oauthor{\bsnm{Lucatelli~Nunes}, \binits{F.}},
\oauthor{\bsnm{Plotkin}, \binits{G.}},
\oauthor{\bsnm{V{\'a}k{\'a}r}, \binits{M.}}:
Unraveling the iterative {CHAD}
(2025).
\url{https://arxiv.org/abs/2505.15002}
\end{botherref}
\endbibitem

%%% 39
\bibitem[\protect\citeauthoryear{Lack}{2012}]{zbMATH06039246}
\begin{barticle}
\bauthor{\bsnm{Lack}, \binits{S.}}:
\batitle{Non-canonical isomorphisms}.
\bjtitle{Journal of Pure and Applied Algebra}
\bvolume{216}(\bissue{3}),
\bfpage{593}--\blpage{597}
(\byear{2012})
\doiurl{10.1016/j.jpaa.2011.07.012}
\end{barticle}
\endbibitem

%%% 40
\bibitem[\protect\citeauthoryear{Lucatelli~Nunes}{2019}]{zbMATH07041646}
\begin{barticle}
\bauthor{\bsnm{Lucatelli~Nunes}, \binits{F.}}:
\batitle{Pseudoalgebras and non-canonical isomorphisms}.
\bjtitle{Applied Categorical Structures}
\bvolume{27}(\bissue{1}),
\bfpage{55}--\blpage{63}
(\byear{2019})
\doiurl{10.1007/s10485-018-9541-3}
\end{barticle}
\endbibitem

%%% 41
\bibitem[\protect\citeauthoryear{Caccamo and Winskel}{2005}]{zbMATH06209910}
\begin{barticle}
\bauthor{\bsnm{Caccamo}, \binits{M.}},
\bauthor{\bsnm{Winskel}, \binits{G.}}:
\batitle{Limit preservation from naturality}.
\bjtitle{Electronic Notes in Theoretical Computer Science}
\bvolume{122},
\bfpage{3}--\blpage{22}
(\byear{2005})
\end{barticle}
\endbibitem

%%% 42
\bibitem[\protect\citeauthoryear{Clementino}{2021}]{zbMATH07377652}
\begin{barticle}
\bauthor{\bsnm{Clementino}, \binits{M.M.}}:
\batitle{$(t,\mathbf{V})$-\textbf{Cat} is extensive}.
\bjtitle{Theory and Applications of Categories}
\bvolume{36},
\bfpage{368}--\blpage{378}
(\byear{2021})
\end{barticle}
\endbibitem

%%% 43
\bibitem[\protect\citeauthoryear{G{\"o}del}{1958}]{godel1958bisher}
\begin{barticle}
\bauthor{\bsnm{G{\"o}del}, \binits{K.}}:
\batitle{{\"U}ber eine bisher noch nicht ben{\"u}tzte erweiterung des finiten
  standpunktes}.
\bjtitle{Dialectica}
\bvolume{12}(\bissue{3--4}),
\bfpage{280}--\blpage{287}
(\byear{1958})
\end{barticle}
\endbibitem

%%% 44
\bibitem[\protect\citeauthoryear{Hyland}{2002}]{Hyland02}
\begin{barticle}
\bauthor{\bsnm{Hyland}, \binits{J.M.E.}}:
\batitle{Proof theory in the abstract}.
\bjtitle{Annals of Pure and Applied Logic}
\bvolume{114}(\bissue{1--3}),
\bfpage{43}--\blpage{78}
(\byear{2002})
\end{barticle}
\endbibitem

%%% 45
\bibitem[\protect\citeauthoryear{de~Paiva}{1989}]{zbMATH04104952}
\begin{bchapter}
\bauthor{\bsnm{Paiva}, \binits{V.C.V.}}:
\bctitle{The {Dialectica} categories}.
In: \bbtitle{Categories in Computer Science and Logic}.
\bsertitle{Contemporary Mathematics},
vol. \bseriesno{92},
pp. \bfpage{47}--\blpage{62}.
\bpublisher{American Mathematical Society},
\blocation{Providence, RI}
(\byear{1989})
\end{bchapter}
\endbibitem

%%% 46
\bibitem[\protect\citeauthoryear{Trotta et~al.}{2023}]{zbMATH07648691}
\begin{barticle}
\bauthor{\bsnm{Trotta}, \binits{D.}},
\bauthor{\bsnm{Spadetto}, \binits{M.}},
\bauthor{\bsnm{Paiva}, \binits{V.}}:
\batitle{Dialectica principles via {G{\"o}del} doctrines}.
\bjtitle{Theoretical Computer Science}
\bvolume{947},
\bfpage{113692}
(\byear{2023})
\doiurl{10.1016/j.tcs.2023.113692}
\end{barticle}
\endbibitem

%%% 47
\bibitem[\protect\citeauthoryear{Diller}{1974}]{diller1974variante}
\begin{barticle}
\bauthor{\bsnm{Diller}, \binits{J.}}:
\batitle{Eine variante zur {Dialectica}-interpretation der heyting-arithmetik
  endlicher typen}.
\bjtitle{Archiv f{\"u}r mathematische Logik und Grundlagenforschung}
\bvolume{16}(\bissue{1--2}),
\bfpage{49}--\blpage{66}
(\byear{1974})
\end{barticle}
\endbibitem

\end{thebibliography}
\end{document}